\documentclass[11pt,a4paper]{amsart}
\usepackage[utf8]{inputenc}
\usepackage[T1]{fontenc}
\usepackage[UKenglish]{babel}
\usepackage[a4paper,margin=1in]{geometry}

\usepackage{amsmath,amsthm,amsfonts,amssymb}
\usepackage{mathrsfs,dsfont}
\usepackage{graphicx}
\usepackage{url}
\usepackage{verbatim}
\usepackage{xcolor}

\makeatletter
\def\namedlabel#1#2{\begingroup
    #2%
    \def\@currentlabel{#2}%
    \label{#1}\endgroup
}
\makeatother

\theoremstyle{plain}
\newtheorem{theorem}{Theorem}[section]
\newtheorem{corollary}[theorem]{Corollary}
\newtheorem{lemma}[theorem]{Lemma}
\newtheorem{proposition}[theorem]{Proposition}

\theoremstyle{definition}
\newtheorem{remark}[theorem]{Remark}
\newtheorem{example}[theorem]{Example}

\numberwithin{equation}{section}

\renewcommand\labelenumi{\textup{\alph{enumi})}}

\renewcommand\theenumi\labelenumi
\makeatletter\renewcommand{\p@enumii}{}\makeatother 

\renewcommand{\leq}{\leqslant}

\renewcommand{\geq}{\geqslant}

\newcommand{\loc}{\mathrm{loc}}

\DeclareMathOperator{\supp}{supp}

\newcommand{\cK}{\mathcal{K}}

\newcommand{\R}{\mathds{R}}

\newcommand{\I}{\mathds{1}}
\newcommand{\Z}{\mathds{Z}}
\newcommand{\pr}{\mathbf{P}}

\newcommand{\ex}{\mathbf{E}}

\begin{document}
\title[Lifshitz tail for L\'{e}vy alloy-type model]
{Lifshitz tail for continuous Anderson models driven by L\'{e}vy operators}
\author[K.~Kaleta]{Kamil Kaleta}
\address[K.~Kaleta]{Faculty of Pure and Applied Mathematics\\ Wroc{\l}aw University of Science and Technology\\ ul. Wybrze{\.z}e Wyspia{\'n}skiego 27, 50-370 Wroc{\l}aw, Poland}
\email{kamil.kaleta@pwr.edu.pl}

\thanks{Research supported by the National Science Center, Poland, grant no.\ 2015/17/B/ST1/01233.}

\author[K.~Pietruska-Pa{\l}uba]{Katarzyna Pietruska-Pa{\l}uba}
\address[K.~Pietruska-Pa{\l}uba]{Institute of Mathematics \\ University of Warsaw
\\ ul. Banacha 2, 02-097 Warszawa, Poland}
\email{kpp@mimuw.edu.pl}

\maketitle

\baselineskip 0.5 cm

\begin{abstract}
We investigate the behavior near zero of the integrated density of states $\ell$ for random Schr\"{o}dinger operators $\Phi(-\Delta) + V^{\omega}$ in $L^2(\mathbb R^d)$, $d \geq 1$, where $\Phi$ is a complete Bernstein function such that for some $\alpha \in (0,2]$, one has $ \Phi(\lambda) \asymp \lambda^{\alpha/2}$, $\lambda \searrow 0$, and $V^\omega (x) = \sum_{\mathbf i\in\Z^d} q_{\mathbf i}(\omega) W(x-\mathbf i)$ is a random nonnegative alloy-type potential with compactly supported single site potential $W$.
 We prove that there are constants $C, \widetilde C,D, \widetilde D>0$ such that
	$$
	-C \leq\liminf_{\lambda \searrow 0} \frac{\lambda^{d/\alpha}}{|\log F_q(D \lambda)|}{\log \ell(\lambda)}
	\qquad \text{and} \qquad \limsup_{\lambda \searrow 0} \frac{\lambda^{d/\alpha}}{|\log F_q(\widetilde D \lambda)|}\log \ell(\lambda) \leq -\widetilde C,
	$$
where $F_q$ is the common cumulative distribution function of the lattice random variables $q_{\mathbf i}$.  In particular, we identify how the behavior of $\ell$ at zero depends on the lattice configuration.
For typical examples of $F_q$ the constants $D$ and $\widetilde D$ can be eliminated from the statement above.  We combine  probabilistic and analytic methods which allow to treat, in  a unified manner,  both  local and non-local kinetic terms such as the Laplace operator, its fractional powers and the quasi-relativistic Hamiltonians.
\bigskip


\smallskip

\end{abstract}

\bigskip\bigskip
\section{Introduction and statement of results}

Let $\Phi$ be a \emph{complete Bernstein function} such that $\lim_{\lambda \searrow 0} \Phi(\lambda)=0$ and let
\begin{equation}\label{eq:oper-def}
H^\omega=\Phi(-\Delta) + V^\omega
\end{equation}
be a random Schr\"{o}dinger operator in $L^2(\R^d)$ with an \emph{alloy-type} potential
\begin{equation}\label{eq:pot-def}
V^\omega (x) = \sum_{\mathbf i\in\Z^d} q_{\mathbf i}(\omega) W(x-\mathbf i),\quad x\in\mathbb R^d,
\end{equation}
where $\{q_{\mathbf i}\}_{\mathbf i\in \Z^d}$ is a sequence of i.i.d.\ nonnegative  and nondegenerate  random variables  over the probability space $(\Omega, \mathcal A, \mathbb Q),$ with cumulative distribution function $F_q(t)=\mathbb Q[q\leq t],$ and $W:\R^d\to\mathbb [0,\infty)$ is a sufficiently regular nonnegative \emph{single-site potential}. The class of kinetic terms $\Phi(-\Delta)$ we consider 
 contains both local operators such as the classical Laplacian $-\Delta$ (for $\Phi(\lambda)=\lambda$) as well as a wide range of non-local pseudo-differential operators which are of great interest in mathematical physics. The most prominent examples in this class are the \textit{fractional Laplacians} $(-\Delta)^{\alpha/2}$ (for $\Phi(\lambda)=\lambda^{\alpha/2}$, $\alpha \in (0,2)$) and the \textit{quasi-relativistic operators} $(-\Delta+m^{2/\vartheta})^{\vartheta/2}-m$ (for $\Phi(\lambda)=(\lambda+m^{2/\vartheta})^{\vartheta/2}-m$, $\vartheta \in (0,2)$, $m>0$) \cite{bib:CMS, bib:J, bib:SSV}.

Under the regularity assumptions {\bf (B)}, {\bf (Q)}  and {\bf (W)} (stated below) on  the  Bernstein function $\Phi$, the distribution function $F_q$, and the single-site potential $W$, respectively, we study the asymptotic behavior of {\em the integrated density of states} (IDS) for the operator $H^\omega$  at the bottom of its spectrum.  The precise definition of IDS  is given in Section \ref{sec:dirichlet}. We will use the same letter $\ell$ to denote both the IDS (a measure) and its cumulative distribution function, i.e. $\ell(\lambda):=\ell([0,\lambda]).$

In 1965, I.M. Lifshitz discovered, on  physical grounds, that the density of states $\ell(\lambda)$  of certain random Hamiltonian $H^{\omega}= H_0 + V^{\omega}$ displays unusually fast decay near the bottom $\lambda_0$ of its spectrum $\sigma(H^{\omega})$: it behaves roughly as $\exp(-c(\lambda-\lambda_0)^{-d/2})$ \cite{bib:Lif}. Since then, such a behavior has been called `the Lifshitz tail'. For $H_0= -\Delta$ (or $-\Delta$ with periodic potential) and for various classes of random potentials $V^{\omega}$  it has been widely studied and rigourously proven in both  continuous (Benderskii and Pastur \cite{bib:BP}, Friedberg and Luttinger \cite{bib:FL}, Luttinger \cite{bib:Lut}, Nakao \cite{bib:Nak}, Pastur \cite{bib:Pas}, Kirsch and Martinelli \cite{bib:KM1, bib:KM2}, Mezincescu \cite{bib:Mez}, Kirsch and Simon \cite{bib:KS}, Kirsch and Veseli\'c \cite{bib:KV}) and  discrete (Fukushima \cite{bib:F}, Fukushima, Nagai and Nakao \cite{bib:FNN}, Nagai \cite{bib:Nag}, Romerio and Wreszinski \cite{bib:RW}, Simon \cite{bib:Sim}) settings (both of these lists are far from being complete). Note in passing that these random Hamiltonians typically exhibit the so-called spectral  localization  (see e.g. Combes and Hislop \cite{bib:Com-His}, Bourgain and Kenig \cite{bib:Bou-Ken},Germinet, Hislop and Klein \cite{bib:Ger-His-Kle}, and the references in these papers). It is known that the Lifshitz singularity is a strong indication for this property to hold and rigorous proofs of localization often use the approximation of the IDS resulting from the Lifshitz asymptotics (see e.g.\ the discussion in the papers by Klopp \cite{bib:Klopp1, bib:Klopp2} and Kirsch and Veseli\'c \cite{bib:KV}).

One of the best studied cases in the classical setting are the Poisson-type potentials.
These random potentials, defined as
    $$
V^{\omega}(x) = \int_{\R^d} W(x-y) \mu^{\omega}({\rm d}y),
$$
where $W$ is a sufficiently regular non-negative profile function and $\mu^{\omega}$ is the Poisson random measure on $\R^d$ with intensity $\nu >0$, were first rigorously investigated in the papers of Nakao \cite{bib:Nak} and Pastur \cite{bib:Pas}.
 This special framework allowed to apply the Donsker-Varadhan large deviations technique to show the following strong statement:
\begin{align}\label{eq:limit-Poiss}
\lim_{\lambda \searrow 0} \lambda^{d/2} \log \ell(\lambda) = - C(d)\,\nu,
\end{align}
where $C(d)$ is an explicit constant.
Later, this result was extended by Okura \cite{bib:Oku} to more general operators $H_0=-L$, where $L$ is a pseudo-differential operator with sufficiently regular Fourier symbol $\Psi(\xi)$ (this class includes $\Delta$ and many other non-local operators $-\Phi(-\Delta)$ studied in the present paper), with Poissonian random potential $V^{\omega}$. It was also the subject of research on less regular spaces such as fractals, see \cite{bib:KPP1,bib:KK-KPP2}.

Back to the lattice (alloy-type) potentials, let us emphasize that the Lifshitz behavior for $H_0=-\Delta$ and non-negative potentials as in \eqref{eq:pot-def} (with $\lambda_0=0$) was typically established in a form somewhat weaker than \eqref{eq:limit-Poiss}, namely
\begin{align}\label{eq:logloglimit}
\lim_{\lambda \searrow 0} \frac{\log |\log \ell(\lambda)|}{\log \lambda}= -\kappa.
\end{align}
Below, this will be referred to as the `loglog statement'. Here $\kappa = d/2$ or $\kappa=d/\beta$, where $\beta >0$ is the parameter describing the decay rate of the single site potential $W$ at infinity (see e.g.\ Kirsch and Martinelli \cite[Theorem 7]{bib:KM2}, Kirsch and Simon \cite[Theorem 1]{bib:KS}).
Kirsch and Martinelli also gave some  sufficient conditions for the existence of constants $C, \widetilde C>0$ such that $e^{-C \lambda^{-d/2}} \leq \ell(\lambda) \leq e^{-\widetilde C \lambda^{-d/2}}$, for $\lambda$ close to zero.

Recently, Kaleta and Pietruska-Pa{\l}uba \cite{bib:KK-KPP-alloy-stable} considered the case of $H_0 = (-\Delta)^{\alpha/2}$, $\alpha \in (0,2]$, and alloy-type potentials $V^{\omega}$ as in \eqref{eq:pot-def} with bounded, compactly supported single-site potentials that are separated from zero in a vicinity of zero, under the assumption that $F_q(\kappa)>0$ for $\kappa>0$. The authors were able to prove that
\begin{align}\label{eq:alloy_atom}
\lim_{\lambda \searrow 0} \lambda^{d/\alpha}\log \ell(\lambda) = -C(d,\alpha) \, \log\frac{1}{F_q(0)} .
\end{align}
Clearly, this limit is finite if and only if $F_q(0)>0$, i.e.\ when the distribution of the lattice random variables has an atom at zero. If this holds, the resulting asymptotics is very close to that in \eqref{eq:limit-Poiss}, known for the Poissonian model (see the more detailed discussion in the introduction of the cited paper). This result shows also that when $F_q(0)=0$, then the limit is infinite, so $\lambda^{d/\alpha}$ is not a correct normalization term for \eqref{eq:alloy_atom}. The correct rate for
$\ell(\lambda)$ should be faster and depend on the behavior of the distribution of the lattice random variables near zero.

This feature motivated our research in the present paper. Our main result, 
 addressing this problem,  is given in Theorem \ref{th:IDS-asymp} below. Interestingly, quite often, such a delicate influence is lost by taking the double logarithm in the `loglog statements' as in \eqref{eq:logloglimit} (cf. Example \ref{ex:distr} (1)-(2) below). As we will see below, a more detailed analysis is required in this case.

Before we pass to the presentation of our results, we introduce the framework assumptions.
They read as follows.

\begin{itemize}	
	\item[\bf (B)] $\Phi$ is a complete Bernstein function satisfying
	\begin{align}\label{eq:basic_ass_BF}
	\lim_{\lambda \to \infty} \frac{\Phi(\lambda)}{\log \lambda} = \infty.
	\end{align}
	and such that there there exist $\alpha\in (0,2],$ $C_1,C_2,\lambda_0>0$ for which
	\begin{equation}\label{eq:assum-phi-close-to-0}
	C_1 \lambda^{ \alpha/2}\leq 	\Phi(\lambda)\leq C_2\lambda^{ \alpha/2},\quad \lambda<\lambda_0.
	\end{equation}	
	
	\item[\bf (Q)] The random variables $ q_{\mathbf i},$ $\mathbf i\in\mathbb Z^d,$ defined on a probability space $(\Omega,\mathcal A, \mathbb Q),$ are independent copies of a non-negative  and non-degenerate (i.e. not equal to a constant a.s.) random variable $q$.  Moreover, denoting by $F_q$  the cumulative distribution function of  $q$, i.e.\ $F_q(\kappa):= \mathbb Q(q \leq \kappa)$,  we assume that  $F_q(\kappa)>0$ for all $\kappa >0$ and  that there  exists $\kappa_0>0$ such that $F_q\big|_{[0,\kappa_0]}$ is continuous (left discontinuity at 0 is permitted).
	
	\item[\bf (W)] The single-site potential $W:\R^d\to \R_+$ has compact support included in $[-M_0,M_0]^d,$
	for certain $M_0\in\Z_+,$ $W\in L^2(\R^d),$ and $\|W\|_2>0.$ Moreover, $W$ belongs to the Kato class of the operator $\Phi(-\Delta).$
\end{itemize}
For precise definitions of Bernstein functions, Kato class etc.\ we refer the reader to Section \ref{sec:prel} below. Note that our assumption {\bf (B)}
covers a wide range of complete Bernstein functions $\Phi$ leading to various important classes of operators $\Phi(-\Delta)$. Some of them are listed in Example \ref{ex:bernstein} in Section \ref{sec:bernstein}. For instance, if $\Phi(\lambda)=\lambda^{\alpha/2}$, $\alpha \in (0,2]$ (the Laplace operator and its fractional powers), then \eqref{eq:assum-phi-close-to-0} clearly holds with the same $\alpha$. For $\Phi(\lambda)=(\lambda+m^{2/\vartheta})^{\vartheta/2}-m$, with $\vartheta \in (0,2)$ and $m>0$ (the quasi-relativistic operators), we have to choose $\alpha=2$ in \eqref{eq:assum-phi-close-to-0}.

The assumption {\bf (W)} is also quite general. It automatically holds for nonnegative $W$'s that are bounded and  of compact support. However,  singular functions $W$ are allowed as well (further details are given in Example \ref{ex:singular} in Section \ref{sec:potentials}).

Our main result reads as follows.

\begin{theorem}\label{th:IDS-asymp}
	Let the assumptions {\bf (B)}, {\bf (Q)} and {\bf (W)} hold  and let $\ell$ be the IDS of the random Schr\"odinger operator defined in \eqref{eq:oper-def}--\eqref{eq:pot-def}. Then there exist constants $C, \widetilde C,D, \widetilde D>0$ such that
	$$
	-C \leq\liminf_{\lambda \searrow 0} \frac{\lambda^{d/\alpha}}{g(D/\lambda)}{\log \ell(\lambda)}
	$$
		and
	$$
		\limsup_{\lambda \searrow 0} \frac{\lambda^{d/\alpha}}{g(\widetilde D/\lambda)}\log \ell(\lambda) \leq -\widetilde C,
	$$
where $g(x)= \log\frac{1}{F_q(D_0/x)}$ and $D_0$ is another nonnegative constant \textup{(}defined in \eqref{eq:c-zero}\textup{)}.
\end{theorem}

Interestingly, the rate for $\log \ell(\lambda)$ can be fully factorized: it depends on the kinetic term $\Phi(-\Delta)$ through $\lambda^{d/\alpha}$ and, separately, on the potential $V^{\omega}$ -- through the function $g(x)$ describing the common distribution of  lattice random variables. To the best of our knowledge, such a general description of the behavior of the IDS at the bottom of the spectrum, involving the dependence of the lattice configuration, was not known before. This is illustrated by Example \ref{ex:distr} below.

Moreover, for functions $g$ with sufficiently regular growth at infinity (see Example \ref{ex:distr} below), the constants   $D,\widetilde D$ in Theorem \ref{th:IDS-asymp} can be eliminated and we arrive at a single statement:
\[-C \leq\liminf_{\lambda \searrow 0} \frac{\lambda^{d/\alpha}}{\log\frac{1}{F_q(\lambda)}}\,\log\ell(\lambda) \leq \limsup_{\lambda \searrow 0} \frac{\lambda^{d/\alpha}}{\log\frac{1}{F_q(\lambda)}}\,\log\ell(\lambda) \leq  -\widetilde C.\]
In particular, if the distribution of the lattice random variables has an atom at zero, i.e.\linebreak $F_q(0)>0$, then the statement simplifies to
\[-C \log\frac{1}{F_q(0)} \leq\liminf_{\lambda \searrow 0} {\lambda^{d/\alpha}}{\log \ell(\lambda)}\leq \limsup_{\lambda \searrow 0} {\lambda^{d/\alpha}}\log \ell(\lambda) \leq -\widetilde C \log\frac{1}{F_q(0)},\]
i.e. we obtain the rate known from the continuous Poisson model (cf. \cite{bib:Oku}).

It is useful to give the following interpretation of our main result (cf. \cite[Remark 3.6(1)]{bib:Szn1} and the comments following \cite[Theorem 1.1]{bib:KK-KPP-alloy-stable}).

\begin{remark} {\textbf{(Interpretation)}}
Suppose $B(r_{\lambda})$ is a  ball with radius $r_{\lambda} \asymp \lambda^{-1/\alpha}$.  Clearly, $|B(r_{\lambda})| \asymp \lambda^{-d/\alpha}$, and due to the condition \eqref{eq:assum-phi-close-to-0} and \cite[Theorem 4.4]{bib:CS} the principal Dirichlet eigenvalue of the operator $\Phi(-\Delta)$ constrained to $B(r_{\lambda})$ is proportional to $\lambda$ as $\lambda \searrow 0$. Then, assuming that $D = \widetilde D = 1$ (this is often the case -- see Example \ref{ex:distr} below) our main result from Theorem \ref{th:IDS-asymp} can be written as
$$
F_q(D_0\lambda)^{C |B(r_{\lambda})|} \leq \ell(\lambda) \leq F_q(D_0\lambda)^{\widetilde C |B(r_{\lambda})|}, \quad \text{for} \quad \lambda \searrow 0.
$$
Since $F_q(D_0\lambda)$ is the probability of the $q_{{\bf i}}$ being not larger than $D_0\lambda$ at any given lattice point, $\ell(\lambda)$ behaves roughly as the probability  that in the ball with ground state eigenvalue comparable to $\lambda$, all random variables $q_{{\bf i}}$'s are smaller than $D_0 \lambda$.
\end{remark}

A direct consequence of Theorem \ref{th:IDS-asymp} is the following 'loglog statement'  which generalizes \eqref{eq:logloglimit}  (see Corollary \ref{coro:loglog}): if $\lim_{x \to \infty} \frac{\log g(x)}{\log x}$ exists, then
$$
\lim_{\lambda \searrow 0}\frac{\log|\log\ell(\lambda)|}{\log \lambda}= -\frac{d}{\alpha}- \lim_{x\to \infty}\frac{\log g(x)}{\log x}.
$$
\smallskip

Our main result is illustrated  by  four different examples of  distribution functions $F_q,$
yielding distinct asymptotics of the IDS near zero (more detailed  discussion of these examples can be found in Section \ref{sec:examples}).
Roughly speaking, the faster the decay of $F_q$ at zero, the faster the decay of $\ell(\lambda)$ at zero as well.

\begin{example} \label{ex:distr} Suppose that the assumptions {\bf (B)} and {\bf (W)} hold. Then there exist constants \linebreak $C, \widetilde C >0$ such that:
\begin{itemize}
\item[(1)] \textbf{atom at zero:} if there exists $\kappa_0 >0$ such that $F_q$ is continuous on $[0,\kappa_0]$ and $F_q(0) > 0$, then
$$
-C\leq \liminf_{\lambda \searrow 0}\lambda^{d/\alpha}\log \ell(\lambda)\leq \limsup_{\lambda \searrow 0}\lambda^{d/\alpha}\log \ell(\lambda)\leq -\widetilde C\quad\mbox{ and }\quad \lim_{\lambda \searrow 0}\frac{\log|\log \ell(\lambda)|}{\log \lambda } =  - \frac{d}{\alpha}.
$$
\item[(2)] \textbf{polynomial decay at zero:} if there exists $\kappa_0 >0$ such that $F_q$ is continuous on $[0,\kappa_0]$ and
$c_1\kappa^{\gamma_1}\leq F_q(\kappa)\leq c_2\kappa^{\gamma_2}$, $\kappa \in [0,\kappa_0]$,
for some $\gamma_1, \gamma_2, c_1, c_2>0$, then
$$
-C\leq \liminf_{\lambda \searrow 0}\frac{\lambda^{d/\alpha}}{\log \lambda}\log \ell(\lambda)\leq \limsup_{\lambda \searrow 0}\frac{\lambda^{d/\alpha}}{\log \lambda}\log \ell(\lambda)\leq -\widetilde C\quad\mbox{ and }\quad \lim_{\lambda \searrow 0}\frac{\log|\log \ell(\lambda)|}{\log \lambda } =  - \frac{d}{\alpha}.
$$
\item[(3)] \textbf{exponential decay at zero:} if $F_q(\kappa) = {\rm e}^{-\frac{1}{\kappa^\gamma}}$, $\kappa>0,$ for some $\gamma >0$, then
$$-C \leq \liminf_{\lambda \searrow 0} \lambda^{\frac{d}{\alpha}+\gamma}\log \ell(\lambda) \leq \limsup_{\lambda \searrow 0} \lambda^{\frac{d}{\alpha}+\gamma}\log \ell(\lambda) \leq -\widetilde C \quad \mbox{and} \quad
\lim_{\lambda \searrow 0} \frac{\log |\log\ell(\lambda)|}{\log \lambda}= -\frac{d}{\alpha}-\gamma.
$$
\item[(4)] \textbf{double-exponential decay at zero:} if $F_q(\kappa)= {\rm e}^{1-{\rm e}^{
\frac{1}{\kappa}}}$, $\kappa>0$, then there exist constants $D_1,D_2 >0$ such that
$$
-C\leq \liminf_{\lambda \searrow 0} \lambda^{d/\alpha}{\rm e}^{ -\frac{D_1}{\lambda}}\log \ell(\lambda) \qquad \mbox{and} \qquad
\limsup_{\lambda \searrow 0} \lambda^{d/\alpha}{\rm e}^{\frac{-D_2}{\lambda}}\log \ell(\lambda)\leq - \widetilde C
$$
and
$$
C\leq\liminf_{\lambda \searrow 0} \lambda \log|\log \ell(\lambda)| \leq \limsup_{\lambda \searrow 0} \lambda \log|\log \ell(\lambda)|\leq \widetilde C.
$$
We call this behavior {\em the super-Lifchitz tail}.
\end{itemize}
\end{example}
As  Example \ref{ex:distr} (3)-(4) above indicates, the contribution coming from the lattice configuration might  not be just the correction term (i.e.\ a lower order term). Actually, its
 order may be polynomial  or even faster, so this term may be the leading one. Such an effect was not observed before. Note also that the distribution functions $F_q$ as in Example \ref{ex:distr} (1)-(2) satisfy the assumptions of the paper by Kirsch and Simon \cite[Theorem 1]{bib:KS}, who established the Lifshitz singularity in the `loglog' form \eqref{eq:logloglimit} for the random Schr\"odinger operators based on Laplacian (i.e.\ for $\Phi(\lambda)=\lambda$). For the case (1),  a version of this results  was first obtained by Kirsch and Martinelli \cite[Theorem 7]{bib:KM2}. Our present work extends and improves these results to the case of compactly supported single-site potentials (see Section \ref{ex:log-rate} for a broader discussion).

Our approach in the present paper is based on a combination of analytic and probabilistic methods.
To make the paper easier accessible to the analytic community,
in Section \ref{sec:prel} we have included a detailed description of subordinate processes  and their evolution semigroups,  Kato classes, Schr\"{o}dinger operators, Feynman-Kac formula etc. The reader familiar with those topics may just give a cursory look at this section and start the lecture from Section \ref{sec:upper}, maybe coming back to the previous section for the notation.

Our argument is constructed as follows. We first find  proper estimates for the Laplace transform $L(t)$ of the IDS (Sections \ref{sec:upper}, \ref{sec:lower}), and then we transform them to the bounds on the IDS itself (Section \ref{sec:tauber}). The proof of the upper bound for $L(t)$ (Theorem \ref{th:upper-short}) is the most demanding part of this work. It consists of several steps which are presented in Sections \ref{sec:trace}, \ref {sec:temple} and \ref{sec:conclusion}, respectively. In the first step, using the stochastic Feynman-Kac representation and monotonicity, we estimate $L(t)$ by the trace of the evolution semigroup of the Schr\"odinger operator $\Phi(\Delta_M) + V^{\omega}_M$, where $\Delta_M$ is the Laplace operator on a torus of given size $M \geq M_0$, and $V^{\omega}_M$ is an alloy-type potential defined for the periodized lattice configuration. With this preparation, in the next step, we are able to apply in our framework a beautiful idea we learned from the papers by Simon \cite{bib:Sim} and Kirsch and Simon \cite{bib:KS}, which is based on an application of the Temple's inequality (Proposition \ref{prop:temple}). In \cite[Proposition 3]{bib:KS} the authors proved that if the ground state eigenvalue of a random Schr\"odinger operator constrained to a box of size $L$ (with Neumann boundary conditions) is smaller than a given number $\lambda>0$, then the number of those lattice random variables in this box which are less than $4\lambda$ is larger than $L^d/2$. In general, our approach in the present paper is different from that in the cited papers (we estimate the Laplace transform directly and we do not employ  the Dirichlet-Neumann bracketing), but we came up with a version of this implication suitable for our setting (Lemma \ref{lem:lambda-on-a-delta}). With a given control level $\delta>0$ and fixed $M,$  we divide all lattice configurations into two disjoint subsets: $\mathcal A_{M,\delta}$ and $\mathcal A_{M,\delta}^c$. In the first set, there is a lot of built-in randomness, which permits to find a proper lower-scaling bound for the ground state eigenvalue of the operator $\Phi(\Delta_M) + V^{\omega}_M$ with $\omega \in \mathcal A_{M,\delta}$.
 The probability of $ \mathcal A_{M,\delta}^c$   is estimated by means of a Bernstein-type estimate for the binomial distribution (Lemma \ref{lem:bernoulli}). We then need to balance the two -- by optimization we find $M=M(t)$ for which both summands are of the same order. This leads us directly to the identification of the correct rate function for $\log L(t)$.

The proof of the lower bound for the Laplace transform $L(t)$ (Theorem \ref{th:lower} of Section \ref{sec:lower}) is more direct. We restrict the integration to the  set of special lattice configurations for  which we can reduce the problem to a careful analysis of the evolution semigroups associated to non-random Schr\"odinger operators $\Phi(-\Delta) + \frac{C}{M^{\alpha}}\sum_{{\mathbf i\in [-M,2M)^d}} W(x-\mathbf i)$ constrained to the box of size $M$ (with Dirichlet conditions).

As the last step, in Section \ref{sec:tauber} we transform the
 statements concerning the asymptotical behavior of $L(t)$  at infinity into statements  for $\ell(\lambda)$ near zero. This is done by an application of a Tauberian-type theorem.
Let us emphasize that the  the asymptotic rates for $\log L(t)$ identified in our Theorems \ref{th:upper-short} and \ref{th:lower} are typically more complicated than $t^{\gamma}$, $\gamma \in (0,1)$  (which are the rates e.g. in the Poissonian case),
as they may comprise also lower order terms (see the more detailed discussion of the specific cases in Section \ref{sec:examples}). This causes additional difficulties as the Tauberian theorems available in the literature did not cover such a case.  Therefore, we had to prove a more general Tauberian theorem which is specialized to work in our present framework (Theorem \ref{eq:taub-lower-assump}).

\subsection*{Convention concerning constants} There are four structural constants of this
paper - $C_1,$ $C_2$ of Assumption {\bf (B)}, $M_0$ of assumption {\bf (W)}, and $D_0$ of
formula \eqref{eq:c-zero}. Their values are kept fixed throughout the paper.
The value of other roman-type constants  (both lower- and upper-case) is not relevant and can change at each appearance.
When we need to keep track of the dependence between technical constants, we number them inside the proofs consecutively as $c_1,c_2, \ldots.$

\section{Bernstein functions and  corresponding Schr\"odinger operators}\label{sec:prel}

As indicated in the Introduction, the approach of this paper is based on a combination of probabilistic and analytic methods. The Schr\"odinger semigroups we consider are represented by
the Feynman--Kac formula with respect to L\'evy processes that are obtained via  subordination (random time change) of the standard Brownian motion in $\R^d$. We start our preparation by giving the necessary preliminaries on Bernstein functions, related stochastic processes, and unbounded operators, then we discuss the class of random potentials studied in this paper. Finally, we introduce the corresponding Schr\"odinger operators and discuss their properties.

\subsection{Bernstein functions and subordinators}

A function $\Phi:(0,\infty) \to [0,\infty)$ is called \emph{completely monotone} if it is smooth and satisfies $(-1)^n \Phi^{(n)}(x) \geq 0$, for every $x >0$ and $n \in \Z_+$. We call $\Phi$ a \emph{Bernstein function} if it is a nonnegative and smooth function with completely monotone derivative. Our standard reference to Bernstein functions,  corresponding operators, and stochastic processes is the monograph \cite{bib:SSV}.

It is known that every Bernstein function $\Phi$ admits the representation
\begin{align} \label{eq:def_Phi}
\Phi(\lambda) = a+b \lambda + \int_{(0,\infty)}(1-{\rm e}^{-\lambda u}) \rho({\rm d}u),
\end{align}
where $a, b \geq 0$ and $\rho$ is a L\'evy measure,  i.e. a nonnegative Radon measure on $(0,\infty)$ such that $\int_{(0,\infty)} (u \wedge 1)\rho({\rm d}u) < \infty$.
A Bernstein function is said to be a \emph{complete Bernstein function} if its L\'evy measure has a completely monotone density with respect to the Lebesgue measure.

Bernstein functions $\Phi$ with $\lim_{\lambda \searrow 0} \Phi(\lambda) = 0$ (i.e.\ $a = 0$) are in one-to-one correspondence with subordinators.
The stochastic process $S=(S_t)_{t \geq 0}$ on a probability space $(\Omega_0, \mathcal{F}, \mathcal{P})$ is called a \emph{subordinator} if it is a nondecreasing  L\'evy process in $\R_+$, i.e.\ a process with c\`adl\`ag paths (right continuous with left limits finite)  starting from $0$, with stationary and independent increments. The laws of $S,$ given by  $\eta_t({\rm d}u):=\mathcal{P}(S_t \in {\rm d}u),$ $t\geq 0,$ form a convolution semigroup of probability measures on $[0,\infty)$ which is uniquely determined by the Laplace transform
\begin{align}\label{eq:laplace}
\int_{[0,\infty)} {\rm e}^{-\lambda u} \eta_t({\rm d}u)  = {\rm e}^{-t \Phi(\lambda)}, \quad \lambda>0,
\end{align}
where the Laplace exponent $\Phi$ is a Bernstein function such that $\lim_{\lambda \searrow 0} \Phi(\lambda) = 0$. The number $b$ and the measure $\rho$ are called the drift term and the L\'evy measure of the subordinator $S,$ respectively.

Under \eqref{eq:basic_ass_BF} we have $\lim_{\lambda \to \infty} \Phi(\lambda) = \infty$, and therefore either $b>0$ or $\int_{(0,\infty)} \rho({\rm d}u) = \infty$. We also easily see from \eqref{eq:laplace} that in this case $\eta_t(\left\{0\right\})=0$, for every $t>0$. The following lemma will be an important tool below. It is based on  standard calculations, but we include here a short proof for the reader's convenience.

\begin{lemma} \label{lem:basic_subord}
For every $\gamma>0$ there is a constant $C=C(\gamma)$ such that
$$
\int_{[0,\infty)} u^{-\gamma} \eta_t({\rm d}u) = C \int_0^{\infty} {\rm e}^{-t \Phi(\lambda^{1/\gamma})} {\rm d} \lambda, \quad t>0.
$$
 Under the assumption \textbf{\textup{(B)}}, for every $t_0>0$ there exists a constant $\widetilde C=\widetilde C(t_0)$ such that
$$
\int_{(0,\infty)} u^{-\gamma} \eta_t({\rm d}u) \leq \widetilde C t^{-2\gamma/\alpha}, \quad t \geq t_0.
$$
In particular, for every $t_0>0$,
$$
\sup_{t \geq t_0} \int_{(0,\infty)} u^{-\gamma} \eta_t({\rm d}u) < \infty.
$$
\end{lemma}
\begin{proof}
First note that by \eqref{eq:laplace} we have ${\rm e}^{-t \Phi(\lambda^{1/\gamma})}=\int_0^{\infty} {\rm e}^{-(\lambda u^{\gamma})^{1/\gamma}} \eta_t({\rm d}u)$, $\lambda, t>0$.
Then, by Fubini-Tonelli and the substitution $\vartheta = u^{\gamma}\lambda$, we get
$$
\int_0^{\infty} {\rm e}^{-t \Phi(\lambda^{1/\gamma})} {\rm d} \lambda =\int_0^{\infty} \int_0^{\infty} {\rm e}^{-(\lambda u^{\gamma})^{1/\gamma}} {\rm d}\lambda \, \eta_t({\rm d}u) = \int_0^{\infty} {\rm e}^{-\vartheta^{1/\gamma}} d\vartheta \, \int_0^{\infty} u^{-\gamma} \eta_t({\rm d}u),
$$
which gives the first equality with $C=\big(\int_0^{\infty} {\rm e}^{-\vartheta^{1/\gamma}} {\rm d}\vartheta\big)^{-1}$.
Moreover, for every fixed $t_0>0$,
$$
\sup_{t \geq t_0} \int_0^{\infty} u^{-\gamma} \eta_t({\rm d}u) = C \sup_{t \geq t_0} \int_0^{\infty} {\rm e}^{-t \Phi(\lambda^{1/\gamma})} {\rm d} \lambda \leq C \int_0^{\infty} {\rm e}^{-t_0 \Phi(\lambda^{1/\gamma})} {\rm d} \lambda.
$$
Assume now \textbf{\textup{(B)}} and fix $t_0>0$. Observe that by \eqref{eq:basic_ass_BF} there exists $\widetilde \lambda_0>\lambda_0$ such that \linebreak $t_0 \Phi(\lambda^{1/\gamma}) \geq 2 \log \lambda$, for $\lambda \geq \widetilde \lambda_0$. Also, by decreasing the constant $C_1>0$ if needed, we may assume that the lower bound in \eqref{eq:assum-phi-close-to-0} holds with $\lambda_0$ replaced with $\widetilde \lambda_0$ (this is possible due to monotonicity and strict positivity of $\Phi$ on $(0,\infty)$). With this in mind,
\begin{align*}
\int_0^{\infty} u^{-\gamma} \eta_t({\rm d}u) = C \int_0^{\infty} {\rm e}^{-t \Phi(\lambda^{1/\gamma})} {\rm d} \lambda & \leq C \left(\int_0^{\widetilde \lambda_0} {\rm e}^{-t \Phi(\lambda^{1/\gamma})} {\rm d} \lambda  +  {\rm e}^{-(t-t_0) \Phi(\widetilde \lambda_0^{1/\gamma})} \int_{\widetilde \lambda_0}^{\infty} {\rm e}^{-t_0 \Phi(\lambda^{1/\gamma})} {\rm d} \lambda \right)\\
& \leq C \left(\int_0^{\widetilde \lambda_0} {\rm e}^{-t A_1 \lambda^{\alpha/(2\gamma)}} {\rm d} \lambda  +  {\rm e}^{-(t-t_0) \Phi(\widetilde \lambda_0^{1/\gamma})} \int_{\widetilde \lambda_0}^{\infty} {\rm e}^{-2 \log \lambda } {\rm d} \lambda \right).
\end{align*}
Using the substitution $\vartheta = t^{2\gamma/\alpha} \lambda$ for the first integral and the fact that there exists a constant $c = c(t_0) >0$ such that $e^{-(t-t_0) \Phi(\widetilde \lambda_0^{1/\gamma} )} \leq c t^{-2\gamma/\alpha}$,  for $t \geq t_0$, we finally get
\begin{align*}
\int_0^{\infty} u^{-\gamma} \eta_t({\rm d}u) \leq C \left(\int_0^{\infty} {\rm e}^{-A_1 \vartheta^{\alpha/(2\gamma)}} {\rm d} \vartheta  + \frac{c}{\widetilde \lambda_0} \right) t^{-2\gamma/\alpha} = \widetilde C t^{-2\gamma/\alpha}.
\end{align*}
This completes the proof.
\end{proof}

\subsection{ Operators $\Phi(-\Delta)$ and subordinate Brownian motions} \label{sec:bernstein}

Denote by $\big\{ G_t: t \geq 0\big\}$ the classical heat semigroup acting on $L^2(\R^d)$, i.e.
$$
G_t f(x) = \int_{\R^d} g_t(x-y) f(y) {\rm d}y, \quad f \in L^2(\R^d), \ t>0,
$$
where $g_t(x)= (4 \pi t)^{-d/2} e^{-|x|^2/4t}$ is the Gauss-Weierstrass kernel. We have $G_t = e^{t \Delta}$, where $\Delta$ is the classical Laplace operator. Recall that it is an unbounded, non-positive definite, self-adjoint operator in $L^2(\R^d)$.
On the probabilistic side, $\big\{ G_t: t \geq 0\big\}$ serves as the transition semigroup of
the standard Brownian motion $Z=(Z_t)_{t \geq 0}$ in $\R^d$, running at twice the usual speed.

Suppose now that $\Phi$ is a Bernstein function such that $\lim_{\lambda \searrow 0} \Phi(\lambda) = 0$ and let $\big\{\eta_t: t \geq 0\big\}$ be the convolution semigroup of measures determined by \eqref{eq:laplace}. With this, we can define
$$
P_t f(x) := \int_{[0,\infty)} G_u f(x) \eta_t({\rm d}u), \quad f  \in L^2(\R^d), \ t \geq 0.
$$
One can check that $P_t$ form a strongly continuous semigroup of bounded self-adjoint operators in $L^2(\R^d)$ which is referred to as  the \emph{subordinate heat semigroup}; under the assumption \eqref{eq:basic_ass_BF} (giving $\eta_t(\left\{0\right\})=0$, $t>0$) all the $P_t$'s, $t>0$, are  integral operators with kernels $p_t(x,y):=p_t(y-x)$, where
$$
p_t(x) = \int_0^{\infty} g_u(x) \eta_t({\rm d}u), \quad  t>0.
$$

Under the full assumption \textbf{\textup{(B)}}, by Lemma \ref{lem:basic_subord}, we obtain that for every $t_0>0$ there exists $C=C(t_0)$ such that
\begin{align}\label{eq:diag_on_pt}
p_t(x) \leq (4 \pi)^{-d/2} \int_0^{\infty} u^{-d/2} \eta_t({\rm d}u) \leq C t^{-d/\alpha}, \quad x \in \R^d, \ \ t \geq t_0.
\end{align}
It is known that $P_t = {\rm e}^{-t\Phi(-\Delta)}$, where $\Phi(-\Delta)$ is the Fourier multiplier with symbol $\Phi(|\xi|^2)$, i.e. the operator
$$
\Phi(-\Delta) f = \mathcal{F}^{-1} \big(\Phi(|\cdot|^2) \mathcal{F} f(\cdot) \big), \quad f \in \mathcal D(\Phi(-\Delta))
$$
with the domain
\[\mathcal D(\Phi(-\Delta)) =\{f\in L^2(\R^d): \Phi(|\xi|^2) \mathcal{F} f(\xi)\in L^2(\R^d)\};\]
$\Phi(-\Delta)$ is an unbounded, non-negative definite,
self-adjoint operator on $L^2(\R^d)$ and  $\mathcal D(\Phi(-\Delta))$ contains $\mathcal D(-\Delta)$. The quadratic form associated with this operator is given by
$$
\mathcal E(f,f) = \int_{\R^d} \Phi(|\xi|^2) (\mathcal{F} f)^2(\xi) {\rm d}\xi, \quad f \in \mathcal D(\mathcal E),
$$
where $f \in \mathcal D(\mathcal E)$ if and only  both of $f(\xi)$ and $\sqrt{\Phi(|\xi|^2)} \mathcal{F} f(\xi)$ are in $L^2(\R^d)$.
Note that the choice $\Phi(\lambda) = b \lambda$, $b>0$, leads to the only pure \emph{local} operator in the class we consider, i.e. the operator $-b \Delta$. Whenever the L\'evy measure $\rho$ in \eqref{eq:def_Phi} is non-zero, the resulting operator  $\Phi(-\Delta)$  is a  \emph{non-local}  integral operator (for $b=0$) or an integro-differential operator (for $b>0$).

The semigroup $\big\{P_t: t \geq 0\big\}$ is the transition semigroup (and $-\Phi(-\Delta)$ is the generator) of a Markov process $X = (X_t)_{t \geq 0}$ which is  determined   by
$$
X_t = Z_{S_t}, \quad t \geq 0.
$$
Such a process is obtained by a random time change  of the Brownian motion $Z$ -- this procedure is called the \emph{subordination}. A new, random, clock of the process is given by the subordinator $S$ (we always assume that $Z$ and $S$ are independent). The process $X$ is referred to as the \emph{subordinate Brownian motion} in $\R^d$. It is an isotropic L\'evy process \cite{bib:Sat} with c\`adl\`ag paths whose L\'evy-Khintchine exponent is equal to $\Phi(|\xi|^2)$. More precisely, we have
$$
\ex_0 {\rm e}^{i \xi \cdot X_t} = {\rm e}^{-t \Phi(|\xi|^2)}, \quad \xi \in \R^d, \ \ t>0.
$$
By $\pr_x$ and $\ex_x$ we denote the probability measure and the corresponding expected value for the process $X$ starting from $x \in \R^d$.
We have $\pr^x(X_t \in A) = \int_A p_t(x,y) {\rm d}y$, $A \in \mathcal B(\R^d)$, $x \in \R^d$, $t>0$, i.e. the kernels $p_t(x,y)$ are transition probability densities of the process $X$. It is important that under \eqref{eq:basic_ass_BF} we also have $\lim_{|\xi| \to \infty} \Phi(|\xi|^2)/ \log |\xi| = \infty$, and it follows from \cite[Lemma 2.1]{bib:KSch2019} that
$(t,x) \mapsto p_t(x)$ is a continuous function on $(0,\infty) \times \R^d$.

Our Assumption {\bf (B)} is satisfied by a wide class of complete Bernstein functions (and corresponding subordinators). Below we discuss only several, the most popular examples. For further examples we refer the reader e.g. to the monograph \cite{bib:SSV}.

\begin{example} \label{ex:bernstein}{
\noindent

\begin{itemize}
\item[(1)] \emph{Pure drift.} Let $\Phi(\lambda)=b \lambda$, $ b>0$. As mentioned above, this leads to the only subordinate Brownian motion with continuous paths -- the \emph{Brownian motion} with speed $b$.
\item[(2)] \emph{$\alpha/2$-stable subordinators.} Let $\Phi(\lambda)=\lambda^{\alpha/2}$, $\alpha \in (0,2)$. The subordination via this subordinator leads to the pure jump \emph{isotropic $\alpha$-stable process}.
\smallskip
\item[(3)] \emph{Mixture of several purely jump stable subordinators.}  In this case,  $\Phi(\lambda)=\sum_{i=1}^n \lambda^{\alpha_i/2}$, $\alpha_i \in (0,2)$, $n \in \Z_+$.
\smallskip
\item[(4)] \emph{$\alpha/2$-stable subordinator with drift.} Let $\Phi(\lambda)=b\lambda + \lambda^{\alpha/2}$, $\alpha \in (0,2)$, $b>0$. Clearly, in this case, $\Phi(\lambda) \approx \lambda$ for $\lambda \to 0^{+}$, and $\Phi(\lambda) \approx \lambda^{\alpha/2}$ for $\lambda \to \infty$.
\item[(5)] \emph{Relativistic $\vartheta/2$-stable subordinator.} Let $\Phi(\lambda)=(\lambda+m^{2/\vartheta})^{\vartheta/2}-m$, $\vartheta \in (0,2)$, $m>0$. The subordination via such a subordinator leads to the so-called \emph{relativistic $\vartheta$-stable process}. Similarly as above, we have $\Phi(\lambda) \approx \lambda$ for $\lambda \to 0^{+}$, and $\Phi(\lambda) \approx \lambda^{\vartheta/2}$ for $\lambda \to \infty$.
\smallskip
\item[(6)] If $S$ is a subordinator with Laplace exponent $\Phi(\lambda)=\lambda^{\alpha/2}[\log(1+\lambda)]^{\beta/2}$, $\alpha \in (0,2)$, $\beta \in (-\alpha, 0)$ or $\beta \in (0,2-\alpha)$, then we see that  both the conditions \eqref{eq:basic_ass_BF} and \eqref{eq:assum-phi-close-to-0} hold as well.

\end{itemize}
}
\end{example}

\medskip

Next, we introduce the  bridge measures of the subordinate process that will be needed in our argument.
For fixed $t>0$ and $x,y \in \R^d,$ the bridge measure $\mathbf P_{x,y}^{t}$ is defined by the following property:
for any $0 <s<t$ and  $A\in \sigma(X_u: u \leq s),$
\begin{equation}\label{eq:bridge}
\pr_{x,y}^{t}[A] = \frac{1}{p_t(x,y)}\mathbf E_x[\mathbf 1_A p_{t-s}(X_s,y)]
\end{equation}
which is  then extended to $s=t$ by weak continuity. The bridge measures can be understood as the laws of the process that starts from $x$ and is conditioned to have $X_t=y,$ $\mathbf P_x-$almost surely.
For more detailed information on Markovian bridges we refer to \cite{bib:Cha-Uri}.

\subsection{The operators and the corresponding subordinate processes on tori} \label{sec:operators_on_tori}

Our argument in the present paper  mostly uses  subordinate semigroups and the related processes on the torus $\mathcal T_M,$ for any $M\in\mathbb Z_+.$  The torus $\mathcal T_M = \mathbb R^d/(M\mathbb Z_+)$ is understood as the box $[0,M)^d$ with reciprocal sides identified.  By $\pi_M$ we denote the canonical projection of $\R^d$ onto $\mathcal T_M$.

Let $\big\{ G^M_t: t \geq 0\big\}$ be the heat semigroup acting on $L^2(\mathcal T_M)$, i.e.
$$
G^M_t f(x) = \int_{\mathcal T_M} g^M_t(x,y) f(y) {\rm d}y, \quad f \in L^2(\mathcal T_M), \quad t>0,
$$
where
$$
g^M_t(x,y):= \sum_{y^{\prime} \in \pi^{-1}_M(y)} g_t(x,y^{\prime}) = \sum_{\mathbf i \in M\Z^d} g_t(x,y+\mathbf i), \ t>0
$$
 is the transition density of the Brownian motion on the torus $\mathcal T_M$ (in this formula, $g_t(x,y) := g_t(y-x)$ denotes the classical Gauss-Weierstrass kernel). One can check that $G^M_t$ form a strongly continuous semigroup of bounded operators in $L^2(\mathcal T_M)$ (the latter fact is an easy consequence of the symmetry $g^M_t(x,y)=g^M_t(y,x)$). The infinitesimal generator of this semigroup (denoted by $\Delta^M$) is an unbounded, self-adjoint operator on $L^2(\mathcal T_M)$.  It is important for our applications that the operators $G^M_t$ have  certain \emph{scaling property}: since $g_{a^2t}(ax) = g_t(x)$, $a>0$, we also have
\begin{align} \label{eq:scaling}
g^{kM}_{k^2t}(kx,ky) = g^{M}_t(x,y), \quad x, y \in \mathcal T_M, \ t>0, \ k \in \Z_+.
\end{align}
The \emph{subordinate heat semigroup on the torus} is defined in the same way as its free counterpart in $\R^d$. For a Bernstein function $\Phi$ such that $\lim_{\lambda \searrow 0} \Phi(\lambda) = 0$ and the convolution semigroup of measures $\big\{\eta_t: t \geq 0\big\}$ determined by \eqref{eq:laplace}, we let
$$
P^M_t f(x) := \int_{[0,\infty)} G^M_u f(x) \eta_t({\rm d}u), \quad f  \in L^2(\mathcal T_M), \ t \geq 0;
$$
$P_t^M$ form a strongly continuous semigroup of bounded self-adjoint operators in $L^2(\mathcal T_M)$. Under the assumption \eqref{eq:basic_ass_BF} all the  $P_t$'s, $t>0$, are integral operators with kernels given by
$$
p^M_t(x,y) = \int_0^{\infty} g^M_u(x,y) \eta_t({\rm d}u),  \quad t>0.
$$
Due to Fubini-Tonelli we have
\begin{align}\label{eq:pM_def}
p^M_t(x,y) = \sum_{y^{\prime} \in \pi^{-1}_M(y)}  \int_0^{\infty}  g_u(x,y^{\prime}) \eta_t({\rm d}u) =  \sum_{y^{\prime} \in \pi^{-1}_M(y)} p_t(x,y^{\prime}) =  \sum_{\mathbf i \in M\Z^d} p_t(x,y+\mathbf i).
\end{align}
We have $P^M_t = {\rm e}^{-t\Phi(-\Delta^M)}$, where  both the operators $\Phi(-\Delta^M)$ and ${\rm e}^{-t\Phi(-\Delta^M)}$ are understood through the spectral  representation of  unbounded self-adjoint operators.

As shown in Lemma \ref{lem:regularity} below, for every fixed $t>0$ both the kernels $g^M_t(x,y)$ and $p^M_t(x,y)$ are bounded functions on $\mathcal T_M \times \mathcal T_M$. Together with the fact that $|\mathcal T_M| = M^d < \infty$, this gives that the operators $G_t^M$ and $P_t^M$ are Hilbert-Schmidt on $L^2(\mathcal T_M)$. In consequence, all the operators considered in this section have purely discrete spectral decompositions. Indeed, for $M=1,2,...,$ the spectrum of the operator $-\Delta^M$ consists of a sequence of eigenvalues
$$
 0 \leq  \mu_1^M < \mu_2^M \leq \mu_3^M \leq ... \to \infty,
$$
each of finite multiplicity, and the corresponding eigenfunctions $\big\{\psi_k^M\big\}_{k=1}^{\infty}$ form a complete orthonormal system in $L^2(\mathcal T_M)$. We have
$$
-\Delta^M \psi^M_k = \mu_k^M \psi^M_k \qquad \text{and} \qquad G_t^M \psi^M_k = {\rm e}^{-t \mu_k^M} \psi^M_k, \ \  t>0, \ \ k =1,2,\ldots,
$$
and due to the conservativeness of the semigroup $\big\{G^M_t: t \geq 0\big\}$,
$$
\mu_1^M = 0 \qquad \text{and} \qquad \psi_1^M \equiv \frac{1}{\sqrt{|\mathcal T_M|}} = \frac{1}{M^{d/2}}.
$$
 One can directly check that  the eigenvalues of the operators $-\Delta^M$ inherit from \eqref{eq:scaling}  the following scaling property:
\begin{align} \label{eq:scaling_eig}
\mu_k^M = M^{-2} \mu_k^1, \quad k = 1,2,3,\ldots.
\end{align}
Due to the spectral theorem, the spectrum of the operator $\Phi(-\Delta^M)$ consists of eigenvalues $\lambda_1^M < \lambda_2^M \leq \lambda_3^M \leq ... \to \infty$ satisfying
\begin{align}\label{eq:eigenvalue_subord}
\lambda_1^M = 0 \qquad \text{and} \qquad \lambda_k^M = \Phi(\mu_k^M), \ \ k = 2,3,\ldots,
\end{align}
and the corresponding eigenfunctions are exactly the same as above. More precisely, we have
\begin{align}\label{eq:eigenvalue_deltaM}
\Phi(-\Delta^M) \psi^M_k = \lambda_k^M \psi^M_k \qquad \text{and} \qquad P_t^M \psi^M_k = {\rm e}^{-t \lambda_k^M} \psi^M_k, \ \  t>0, \ \ k =1,2,\ldots.
\end{align}

Our present work requires additional  regularity  properties of the kernels $p^M_t(x,y)$ such as continuity, boundedness, and on-diagonal estimates,  gathered in the following Lemma.
\begin{lemma} \label{lem:regularity}
The following hold.
\begin{itemize}
\item[(1)] There exists a constant $C >0$ such that for every $M, n \in \Z_+$, $x,y \in [0,M)^d$ and $t>0$ we have
$$
\sum_{\mathbf i \in M\Z^d \atop \mathbf i \notin [-nM,nM]^d } g_t(x,y+\mathbf i) \leq \frac{C}{M^{d}} \left(\frac{Mn}{\sqrt{t}} \vee 1\right)^{d-2} {\rm e}^{-\frac{1}{16}\left(\frac{Mn}{\sqrt{t}} \vee 1\right)^2}.
$$
In particular, the series defining the kernel $g^M_t(x,y)$ is uniformly convergent in $(t,x,y)$ on every cuboid $[u,v] \times [0,M)^d \times [0,M)^d$,  $0 < u < v < \infty$, and there exists a universal constant $\widetilde C >0$ such that for every $M \in \Z_+$ we have
$$
g^M_t(x,y) \leq \widetilde C \left(t^{-d/2} \vee M^{-d}\right), \qquad x,y \in \mathcal T_M, \ t>0.
$$
\item[(2)] Under the assumption \eqref{eq:basic_ass_BF} the function $(t,x,y) \mapsto p_t^M(x,y)$ is continuous on \linebreak $(0,\infty) \times \mathcal T_M \times \mathcal T_M$. Moreover, there exists a universal constant $C>0$ such that for every $M \in Z_+$ we have
$$
p^M_t(x,y) \leq C \left(\int_0^{\infty} {\rm e}^{-t \Phi(\lambda^{2/d})}{\rm d}  \lambda \vee M^{-d}\right), \qquad x,y \in \mathcal T_M, \ t>0.
$$
In particular, $p_t^M(x,y)$ is bounded on every cuboid $[t_0,\infty) \times \mathcal T_M \times \mathcal T_M$, $t_0>0$.
\end{itemize}
\end{lemma}

\begin{proof}
\noindent
(1) For every $x,y \in [0,M)^d$, $t>0$ and $M, n \in \Z_+$ we have
\begin{eqnarray*}
\sum_{\mathbf i \in M\Z^d \atop \mathbf i \notin [-nM,nM]^d } g_t(x,y+\mathbf i) & \leq& \sum_{k \geq n} \frac{(2(k+1))^d-(2k)^d}{(4\pi t)^{d/2}} {\rm e}^{-\frac{(kM)^2}{4t}} \\
                                     & \leq &\frac{c}{\sqrt{t} M^{d-1}}\sum_{k \geq n} \left(\frac{kM}{\sqrt{t}}\right)^{d-1}{\rm e}^{-\left(\frac{(k+1)M}{4\sqrt{t}}\right)^2} \\
																		 & \leq &\frac{c}{\sqrt{t} M^{d-1}}\int_n^{\infty}\left(\frac{Mx}{\sqrt{t}}\right)^{d-1}{\rm e}^{-\frac{1}{ 16} \left(\frac{Mx}{\sqrt{t}}\right)^2} {\rm d}x,
\end{eqnarray*}
with an absolute constant $c>0$. By substitution, the latter expression is equal to
$$
\frac{c}{M^{d}}\int_{\frac{Mn}{\sqrt{t}}}^{\infty}y^{d-1} {\rm e}^{-\frac{y^2}{16}} {\rm d}y.
$$
Using the elementary estimate
$$
\int_a^{\infty} y^{d-1} {\rm e}^{-\frac{y^2}{ 16}} {\rm d}y \leq c_1 (a \vee 1)^{d-2} {\rm e}^{-\frac{1}{ 16}(a \vee 1)^2}, \quad a >0,
$$
where $c_1>0$ is a uniform constant, we finally get
$$
\sum_{\mathbf i \in M\Z^d \atop \mathbf i \notin [-nM,nM]^d } g_t(x,y+\mathbf i) \leq \frac{c_2}{M^{d}} \left(\frac{Mn}{\sqrt{t}} \vee 1\right)^{d-2} {\rm e}^{-\frac{1}{ 16}\left(\frac{Mn}{\sqrt{t}} \vee 1\right)^2},
$$
which is exactly the first assertion of part (1). The uniform convergence follows directly from this uniform bound for the tail of the series. To prove the other assertion of (1), we write
$$
g_t^M(x,y) \leq c_3 t^{-d/2} + \sum_{\mathbf i \in M\Z^d \atop \mathbf i \notin [-M,M]^d } g_t(x,y+\mathbf i).
$$
The second term can be easily estimated by using the bound proven above with $n=1$. We have two cases. If $\sqrt{t} \geq M$, then $g_t^M(x,y) \leq c_3 t^{-d/2}+c_2M^{-d} \leq (c_2+c_3) M^{-d}$. If  $\sqrt{t} < M$, then
$$
\left(\frac{M}{\sqrt{t}} \vee 1\right)^{d-2} {\rm e}^{-\frac{1}{ 16}\left(\frac{M}{\sqrt{t}} \vee 1\right)^2} = \left(\frac{M}{\sqrt{t}}\right)^{d-2} {\rm e}^{-\frac{1}{ 16}\left(\frac{M}{\sqrt{t}}\right)^2} \leq c_4,
$$
and, similarly as above, $g_t^M(x,y) \leq c_3 t^{-d/2}+c_2c_4M^{-d} \leq (c_2c_4+c_3) t^{-d/2}$. This implies the second estimate in (1).

\smallskip

\noindent
(2) We first show the estimate and the boundedness. By the upper estimate for the kernel $g_t^M(x,y)$ proven above, for $M \in \Z_+$, $x, y \in \mathcal T_M,$ and $t>0$, we have
$$
p^M_t(x,y) \leq c_4 \left(\int_0^{M^2} u^{-d/2} \eta_t({\rm d}u) + M^{-d} \eta_t[M^2,\infty)\right) \leq c_4 \left(\int_0^{\infty} u^{-d/2} \eta_t({\rm d}u) + M^{-d}\right).
$$
We then derive from Lemma \ref{lem:basic_subord} that
$$
p^M_t(x,y) \leq c_5 \left(\int_0^{\infty} {\rm e}^{-t \Phi(\lambda^{2/d})} {\rm d}\lambda +M^{-d} \right),
$$
and, for every $t_0>0$,
$$
\sup_{(t,x,y) \in [t_0,\infty) \times \mathcal T_M \times \mathcal T_M} p^M_t(x,y) < \infty.
$$
We now prove the continuity. Since the function $(t,x,y) \mapsto p_t(x,y)$ is continuous on \linebreak $(0,\infty) \times \R^d \times \R^d$, it is enough to justify that the series
$$
\sum_{\mathbf i \in M\Z^d} p_t(x,y+\mathbf i)
$$
is uniformly convergent on every cuboid $[t_0, t_1] \times [0,M)^d \times [0,M)^d$, $0<t_0 < t_1 < \infty$. We only need to prove that the tail
$$
\sum_{\mathbf i \in M\Z^d \atop \mathbf i \notin [-nM,nM]^d } p_t(x,y+\mathbf i)
$$
goes to zero as $n \to \infty$, uniformly in $(t,x,y) \in [t_0, t_1] \times [0,M)^d \times [0,M)^d$. Using the tail estimate from part (1), Fubini-Tonelli an the fact that $c_6:= \sup_{r \geq 1} r^{ d-2} {\rm e}^{- r^2/16} < \infty$, we get
\begin{align*}
\sum_{\mathbf i \in M\Z^d \atop \mathbf i \notin [-nM,nM]^d } & p_t(x,y+\mathbf i) \leq  \frac{c_7}{M^{d}} \int_0^{\infty}\left(\frac{Mn}{\sqrt{u}} \vee 1\right)^{d-2} {\rm e}^{-\frac{1}{16}\left(\frac{Mn}{\sqrt{u}} \vee 1\right)^2} \eta_t({\rm d}u) \\
                                                         & \leq \frac{c_7}{M^{d}} \left(\int_{(0,n]}\left(\frac{Mn}{\sqrt{u}}\right)^{d-2} {\rm e}^{-\frac{1}{16}\left(\frac{Mn}{\sqrt{u}}\right)^2} \eta_t({\rm d}u) + c_6 \eta_t\big(n, \infty\big)\right) \\
																												 & \leq \frac{c_7}{M^2} n^{ d-1} {\rm e}^{-\frac{n}{16}} \int_{(0,n]} u^{-d/2} \eta_t({\rm d}u) + c_6c_7 \eta_t(n,\infty).
\end{align*}
Now, by Lemma \ref{lem:basic_subord}, the integral $\int_0^\infty u^{-d/2} \eta_t({\rm d}u)$ is uniformly bounded for $t \geq t_0$. Moreover, by following the argument in \cite[Lemma 2.2]{bib:KK-KPP1}, we obtain
$$
\int_0^n {\rm e}^{-\lambda u} \eta_t(u,\infty) {\rm d}u \leq \frac{t \Phi(\lambda)}{\lambda}, \quad \lambda, t>0,
$$
which yields

$$
\eta_t(n,\infty) (1-{\rm e}^{-\lambda n}) \leq t \Phi(\lambda), \quad \lambda, t>0.
$$
By taking $\lambda = 1/n$, we get $\eta_t(n,\infty) \leq (1-{\rm e}^{-1})^{-1} t \Phi(1/n) \leq (1-{\rm e}^{-1})^{-1} t_1 \Phi(1/n)$, whenever $t \leq t_1$. This implies the claimed uniform convergence, completing the proof of the lemma.

\end{proof}

The semigroup $\big\{P^M_t: t \geq 0\big\}$ determines a conservative Markov process $(X^M_t)_{t \geq 0}$ on
the  torus $\mathcal T_M$. If we denote by $\mathbf P_x^M$  the measure concentrated on trajectories that start from  $x \in \mathcal T_M$, then
$$
\mathbf P_x^M(X^M_t \in A) = P^M_t \I_A(x) = \int_A p^M_t(x,y) {\rm d}y, \quad A \in \mathcal B(\mathcal T_M), \ x \in \mathcal T_M, \ t >0.
$$
It is a symmetric Feller process with continuous and bounded transition probability densities $p^M_t(x,y)$. Due to \eqref{eq:pM_def} this process can be identified pathwise as
$$
X^M_t = \pi_M(X_t), \quad t>0,
$$
where $(X_t)_{t \geq 0}$ is the subordinate Brownian motion in $\R^d$, introduced in the previous section. Throughout the paper we call this process the \emph{subordinate Brownian motion on the torus $\mathcal T_M$}.

For given $t\geq 0$ and   $x,y\in\mathcal T_M,$  the bridge measure of the process $X^M,$ conditioned to have $ X_t^M=y,$
$\mathbf P_x^M-$almost surely,   is defined by a relation similar to \eqref{eq:bridge}. These measures are denoted  by $\mathbf P_{x,y}^{M,t}.$

The bridge measures for the
process in $\R^d$ and on the torus $\mathcal T_M$ are related through the
following identity.

\begin{lemma}
	\label{lm:rotation}
	 For every $t>0$, $x,y \in \R^d$, $M=1,2,...$ and any set $A \in \mathcal B(D([0,t],\mathcal T_M)$ we have
	\begin{equation}\label{eq:rot1}
	p^M_t(\pi_M(x),\pi_M(y))\mathbf P^{M,t}_{\pi_M(x),\pi_M(y)}[A]=
	\sum_{y'\in\pi_M^{-1}(\pi_M(y))} p_t(x,y')\mathbf
	P^t_{x,y'}[\pi_M^{-1}(A)]\end{equation}	
 ($D([0,t],\mathcal T_M))$ is the Skorohod space).
\end{lemma}
This statement  is readily seen for cylindrical sets and then
extended to the desired range of $A$'s by the Monotone Class Theorem.
Its fractal counterpart was discussed in \cite[Lemma 2.6]{bib:KK-KPP1}.

\subsection{Random  Anderson (alloy-type) potentials} \label{sec:potentials}
Our approach in the present paper allows us to study the alloy-type random fields
\begin{align} \label{eq:alloy}
V^{\omega}(x) = \sum_{{\bf i} \in \Z^d} q_{{\bf i}}(\omega) W(x-{\bf i}), \quad x \in \R^d,
\end{align}
with  possibly  singular single-site potentials $W$ of bounded support which are in Kato classes corresponding to the operators considered. The main part of our argument is based on an application of  certain periodization of such potentials: for given $M\geq 1$ we define
\begin{eqnarray}\label{eq:szn-per-of-V}
V_M^{\omega}(x)&:= &\sum_{{\bf i}\in[0, M)^d} \left(q_{\bf i}(\omega) \sum_{{\bf i'}\in \pi_{M}^{-1}({\bf i})} W(x-{\bf i'})\right) \nonumber\\
&=& \sum_{ {\bf i} \in \Z^d} q_{\pi_M({\bf i})} W(x-{\bf i}),\quad x\in\R^d.
\end{eqnarray}
This means that we first periodize the lattice random variables $\{q_{\mathbf i}\}_{\mathbf i\in \Z^d}$ with respect to $\pi_M$, and then, based on that, we construct a new random potential which is also periodic in the usual sense: $V_M^{\omega}(x+{\bf i}) = V_M^{\omega}(x)$, ${\bf i} \in M \Z^d$.  We call it  the {\em Sznitman-type periodization} of $V^{\omega}$.
For simplicity, we will use the same letter for the restriction of this potential to $\mathcal T_M.$

Recall that the Kato class $\mathcal K$ associated with the operator $\Phi(-\Delta)$ consists of those Borel functions $f:\R^d \to \R$ for which
$$
\lim_{t \searrow 0} \sup_{x \in \R^d} \int_0^t P_s|f|(x)\,{\rm d}s = 0.
$$
Similarly, a Borel function $f:\mathcal T_M \to \R$ belongs to the Kato class $\mathcal K^M$ associated with $\Phi(-\Delta^M)$ if
$$
\lim_{t \searrow 0} \sup_{x \in \mathcal T_M} \int_0^t P^M_s|f|(x)\,{\rm d}s = 0.
$$
Moreover,  we say that a Borel function $f$ belongs to the local Kato class $\mathcal K_{\loc}$ if its restriction to an arbitrary bounded Borel subset of  $\R^d$ is in $\mathcal K$.
Note that the torus $\mathcal T_M$ is a compact space and so the local Kato class for $\mathcal T_M$ would agree with $\mathcal K^M$. Therefore there is no need to define it separately. One can check that $L^{\infty}(\R^d) \subset \mathcal K$ and $L^{\infty}_{\loc}(\R^d) \subset \mathcal K_{\loc}$. Moreover, $\mathcal K_{\loc} \subset L^1_{\loc}(\R^d)$ and $\mathcal K^M \subset L^1(\mathcal T_M) = L^1_{\loc}(\mathcal T_M)$.

We now show that the alloy-type random potentials $V^{\omega}$ and $V^{\omega}_M$ inherit the Kato-regularity from their profiles $W$.

\begin{proposition} \label{prop:Kato_regularity} Let $W \in \mathcal K$, $W \geq 0$, be of bounded support and let the assumption \textbf{\textup{(Q)}} hold.
Then, for every $M \in \Z_+$ and $\omega \in \Omega$, we have
\begin{itemize}
\item[(1)] $V^{\omega} \in \mathcal K_{\loc}$,
\item[(2)] $V^{\omega}_M \in \mathcal K_{\loc}$,
\item[(3)] $V^{\omega}_M \in \mathcal K^M$.
\end{itemize}
\end{proposition}

\begin{proof}
Fix $\omega \in \Omega$ and suppose $\supp W \subset [-M_0,M_0]^d$, for some $M_0 \in \Z_+$.

\medskip

\noindent
(1) Denote $V^{n,\omega} = \I_{[-n,n]^d} V^{\omega}$, $n \in \Z_+$. Observe that
$$
V^{n,\omega}(x) \leq \sum_{{\bf i} \in \Z^d \cap [-M_0-n,M_0+n]^d} q_{{\bf i}}(\omega) W(x-{\bf i}), \quad x \in \R^d.
$$
We have
\begin{align*}
\int_0^t P_s V^{n,\omega}(x)\,{\rm d}s & \leq \sum_{{\bf i} \in \Z^d \cap [-M_0-n,M_0+n]^d} q_{{\bf i}}(\omega) \int_0^t P_s W (x-{\bf i})\,{\rm d}s \\
                           & \leq \sup_{y \in \R^d} \int_0^t P_s W (y) \left(\sum_{{\bf i} \in \Z^d \cap [-M_0-n,M_0+n]^d} q_{{\bf i}}(\omega)\right)\,{\rm d}s.
\end{align*}
The sum on the right hand side  has finitely many terms and $W \in \mathcal K$. Therefore  by taking the supremum over $x \in \R^d$ on the left hand side and then  letting $t \searrow 0$, we get that
$$
\sup_{x \in \R^d} \int_0^t P_s V^{n,\omega}(x) \to 0,
$$
for arbitrary $n \in \Z_+$.
Hence $V^{\omega} \in \mathcal K_{\loc}$.

\medskip

\noindent
(2) The proof is a minor modification of that of (1) as we only need to replace $q_{{\bf i}}(\omega)$ with $q_{\pi_M({\bf i})}(\omega)$ in the sum defining the potential.

\medskip

\noindent
(3) Fix $M \in \Z_+$. By the definition of the operators $P_t^M$ and the potential $V^{\omega}_M$, for every $x \in \mathcal T_M$ and $s>0$ we have
\begin{align*}
P_s^M V^{\omega}_M(x) & = \int_{\mathcal T_M} p^M_s(x,y) V^{\omega}_M(y) {\rm d}y \\
                      & = \int_{[0,M)^d} \left(\sum_{{\bf i} \in M\Z^d} p_t(x,y+{\bf i}) \right)V^{\omega}_M(y) {\rm d}y \\
											& = \int_{[0,M)^d} \left(\sum_{{\bf i} \in M\Z^d \atop {\bf i} \in [-M,M]^d} p_t(x,y+{\bf i}) \ + \sum_{{\bf i} \in M\Z^d \atop {\bf i} \notin [-M,M]^d} p_t(x,y+{\bf i}) \right)V^{\omega}_M(y) {\rm d}y \\
											& = P_t \big(\I_{[-M,M]^d}V^{\omega}_M\big)(x)
											       + \int_{[0,M)^d} \left(\sum_{{\bf i} \in M\Z^d \atop {\bf i} \notin [-M,M]^d} p_t(x,y+{\bf i}) \right)V^{\omega}_M(y) {\rm d}y.
\end{align*}
By Fubini-Tonelli and the tail estimate in Lemma \ref{lem:regularity} (1)
$$
\sum_{{\bf i} \in M\Z^d \atop {\bf i} \notin [-M,M]^d} p_t(x,y+{\bf i})
= \int_0^{\infty} \left(\sum_{{\bf i} \in M\Z^d \atop {\bf i} \notin [-M,M]^d} g_t(x,y+{\bf i}) \right) \eta_t({\rm d}u) \leq cM^{-d},
$$
which gives
$$
P_s^M V^{\omega}_M(x) \leq  P_t \big(\I_{[-M,M]^d}V^{\omega}_M\big)(x) + cM^{-d} \int_{[0,M)^d}V^{\omega}_M(y) {\rm d}y.
$$
By part (2) we have $V^{\omega}_M \in \mathcal K_{\loc}$. In particular, $V^{\omega}_M \in L^1_{\loc}(\R^d)$. Hence
$$
\sup_x \int_0^t P_s^M V^{\omega}_M(x)\,{\rm d}s \leq \sup_x \int_0^t P_t \big(\I_{[-M,M]^d}V^{\omega}_M\big)(x)\,{\rm d}s + c t M^{-d} \int_{[0,M)^d}V^{\omega}_M(y)\, {\rm d}y \longrightarrow 0,
$$
as $t \searrow 0$. This means that $V^{\omega}_M \in \mathcal K^M$.
\end{proof}

As mentioned in the introduction, every bounded function with compact support is automatically in the Kato class $\cK$. We now provide  examples of singular functions from $\cK$.

\begin{example} \label{ex:singular}
Let $\Phi(\lambda) = \lambda^{\alpha/2}$, $\alpha \in (0,2]$ (i.e. we either consider the Laplace operator $-\Delta$ or the fractional Laplace  operators $(-\Delta)^{\alpha/2}$, $\alpha \in (0,2)$). For simplicity, assume additionally that $\alpha < d \in \Z_{+}$. It is known (see e.g.\ \cite{bib:BB}) that in this case
$$
f \in \cK \qquad \Longleftrightarrow \qquad \lim_{r \searrow 0} \sup_{x \in \R^d} \int_{|x-y|<r} \frac{f(y)}{|x-y|^{d-\alpha}} {\rm d}y = 0.
$$
The same is true for $\Phi(\lambda) = (\lambda+m^{2/\alpha})^{\alpha/2}-m$, $\alpha \in (0,2)$, $m>0$, i.e. for the quasi-relativistic operators.
If we now take $W(y):= \I_{B(0,1)}(y) |y|^{-\beta}$, $\beta >0$, then we see that   $W \in \cK$ if and only if $\beta < \alpha$.
 In view of the assumption {\bf (W)} it is also instructive to  verify   that $W \in L^2(\R^d)$ if and only if $\beta < d/2$. In particular, $W \in \cK \cap L^2(\R^d)$ if and only if $\beta < \alpha \wedge d/2$.
\end{example}

\smallskip

This example indicates that the intersection $\cK \cap L^2(\R^d)$ is typically a fairly non-trivial function space, but in general there are no inclusions between $\cK$ and $L^2(\R^d)$.

\subsection{Schr\"odinger operators and the Feynman--Kac formula}

We now  introduce the class of random Schr\"odinger operators based on $\Phi(-\Delta)$ and $-\Phi(-\Delta^M)$, and we discuss their spectral properties.
Our standard reference here will be the monograph of Demuth and van Casteren \cite{bib:DC} which is concerned with the spectral theory of self-adjoint Feller operators.

In the previous sections we have verified that the \emph{subordinate semigroups} $\big\{P_t:t \geq 0\big\}$ and $\big\{P^M_t:t \geq 0\big\}$, determined by the kernels $p_t(x,y)$ and $p^M_t(x,y)$, respectively, satisfy the \emph{basic assumptions of spectral stochastic analysis} (BASSA in short)  and in consequence  the operators $-\Phi(-\Delta)$ and $-\Phi(-\Delta^M)$ are \emph{(free) Feller generators} \cite[Assumptions A1-A4 and Definition 1.3 in Section 1.B]{bib:DC}.

Throughout this section we assume that $V^{\omega}$ and $V^{\omega}_M$ are random alloy-type potentials given by \eqref{eq:alloy} and \eqref{eq:szn-per-of-V}, constructed
for a compactly supported and nonnegative single-site potential $W \in \mathcal K$ and lattice random variables $\big\{q_{\bf i}\big\}_{{\bf i} \in \Z^d}$ satisfying the assumption \textbf{(Q)}. Thus, by Proposition \ref{prop:Kato_regularity}, we have  $V^{\omega} \in \mathcal K_{\loc}$ and $V^{\omega}_M \in \mathcal K^M$, for every realization of lattice configuration. This allows us to define the random Schr\"odinger operators
$$
H^{\omega} = \Phi(-\Delta)+V^{\omega} \qquad \text{and} \qquad H^{\omega}_M = \Phi(-\Delta^M)+V^{\omega}_M
$$
as positive self-adjoint operators on $L^2(\R^d)$ and $L^2(\mathcal T_M)$, respectively \cite[Theorem 2.5]{bib:DC}. It is decisive for this work that the evolution semigroups of these operators can be represented probabilistically with respect to subordinate processes $(X_t)_{t \geq 0}$ and $(X^M_t)_{t \geq 0}$. More precisely, the following Feynman--Kac formulas hold:
$$
P_t^{V^{\omega}} f(x): ={\rm e}^{-tH^{\omega}} f(x) =  \ex_x\left[{\rm e}^{-\int_0^t V^{\omega}(X_s) {\rm d}s} f(X_t)\right], \quad f \in L^2(\R^d), \quad t>0,
$$
and
$$
P_t^{M,V^{\omega}_M} f(x) :={\rm e}^{-tH^{\omega}_M} f(x) =  \ex_x^M\left[{\rm e}^{-\int_0^t V^{\omega}_M(X^M_s) {\rm d}s} f(X^M_t)\right], \quad f \in L^2(\mathcal T_M), \quad t>0.
$$
Both $P_t^{V^{\omega}}$ and  $P_t^{M,V^{\omega}_M}$, $t>0$, are integral operators with bounded and symmetric kernels
$$
p^{V^{\omega}}_t(x,y) = p_t(x,y) \ex_{x,y}^t\left[{\rm e}^{-\int_0^t V^{\omega}(X_s) {\rm d}s}\right],
$$
and
\begin{equation}\label{eq:FK-torus}
p^{M,V^{\omega}}_t(x,y) = p^M_t(x,y) \ex_{x,y}^{M,t}\left[{\rm e}^{-\int_0^t V^{\omega}_M(X^M_s) {\rm d}s}\right],
\end{equation}
where $\ex_{x,y}^t$ and $\ex_{x,y}^{M,t}$ are expected values with respect to bridge measures $\pr_{x,y}^t$ and $\pr_{x,y}^{M,t}$ introduced in  previous sections.

Observe that $p^{M,V^{\omega}}_t(x,y) \leq p^M_t(x,y)$ and recall that for every fixed $t>0$ the kernel $p^M_t(x,y)$ is a bounded function on $\mathcal T_M \times \mathcal T_M$. Since $|\mathcal T_M| = M^d < \infty$, this gives that the operators $P_t^{M,V^{\omega}_M}$, $t>0$, are Hilbert-Schmidt in $L^2(\mathcal T_M)$. This implies that the spectrum of the Schr\"odinger operator $H^{\omega}_M$ is discrete -- it consists of a sequence of eigenvalues
$$ 0\leq
\lambda_1^{M, V^{\omega}} < \lambda_2^{M, V^{\omega}} \leq \lambda_3^{M, V^{\omega}} \leq ... \to \infty
$$
with finite multiplicities, and the corresponding eigenfunctions $\big\{\psi_k^{M, V^{\omega}}\big\}_{k=1}^{\infty}$ form a complete orthonormal system in $L^2(\mathcal T_M)$.

\subsection{Dirichlet Schr\"odinger operators and the integrated density of states} \label{sec:dirichlet}

Denote by $H^{\omega}_\Lambda$  the operator $H^{\omega}$ constrained to a bounded, nonempty region $\Lambda\subset \R^d$ (we consider Dirichlet conditions on $\Lambda^c$ in the non-local case and on $\partial \Lambda$ in the local case) and let $\big\{{\rm e}^{-tH^{\omega}_{\Lambda}}; t\geq 0\big\}$ be its evolution semigroup on $L^2(\Lambda)$. Then we have the  following Feynman--Kac formula:
$$
{\rm e}^{-tH^{\omega}_{\Lambda}} = P_t^{V^{\omega},\Lambda} f(x):= \ex_x\left[{\rm e}^{-\int_0^t V^{\omega}(X_s) {\rm d}s} f(X_t); t<\tau_{\Lambda}\right], \quad f \in L^2(\Lambda, {\rm d}x), \quad t>0.
$$
Here $\tau_{\Lambda}:=\inf\{t\geq 0: \, X_t\notin \Lambda\}$ denotes the first exit time of the process from the domain $\Lambda$. All the $P_t^{V^{\omega},\Lambda}$, $t>0$, are integral operators with bounded and symmetric kernels
\begin{align}
\label{def:sem-dir-kernel-bridge}
p^{V^{\omega}, \Lambda}_t(x,y) = p_t(x,y) \ \ex^t_{x,y}\left[{\rm e}^{-\int_0^t V^{\omega}(X_s) {\rm d}s} ; t<\tau_{\Lambda}\right].
\end{align}
Again, since $|\Lambda|<\infty$, the operators $P_t^{V^{\omega},\Lambda}$, $t>0$, are Hilbert-Schmidt. In particular, there exists a complete orthonormal system, consisting of eigenfunctions of the operator $H^{\omega}_{\Lambda}$. The corresponding eigenvalues satisfy $0\leq \lambda_1^{V^{\omega}}(\Lambda) < \lambda_2^{V^{\omega}}(\Lambda)\leq \lambda_3^{V^{\omega}}(\Lambda) \leq \ldots \to \infty$; each $\lambda_k^{V^{\omega}}(\Lambda)$ is of finite multiplicity and the ground state eigenvalue $\lambda_1^{V^{\omega}}(\Lambda)$ is simple.

We are now in a position to give  the formal definition of the IDS. For a given bounded domain $\Lambda\subset \R^d$, let
\[\ell_{\Lambda}^{\omega}(\cdot)=\frac{1}{|\Lambda|}\sum_{k=1}^\infty\delta_{\lambda_k^{V^{\omega}}(\Lambda)}
(\cdot)\]
be the counting measure on the spectrum of $H_{\Lambda}^{\omega},$ normalized by the volume.
Under the assumption \textbf{(Q)}, the random alloy-type potential $V^{\omega}$ is stationary with respect to $\Z^d$. Therefore if we restrict our attention
to sets $\Lambda$ composed of  unit cubes with  vertices in $\Z^d,$ then
it follows from the maximal ergodic theorem
(see e.g. \cite[Remark VI.1.2]{bib:Car-Lac})  that the measures
$\ell_{\Lambda}^{\omega}$ converge vaguely, as $\Lambda \nearrow \R^d,$ to a nonrandom measure $\ell,$ which is
called the integrated density of states of $H^\omega.$ The  vague convergence of $\ell_{\Lambda}^{\omega}$ when $\Lambda \nearrow \R^d$ amounts to the convergence of their Laplace transforms
\begin{eqnarray*}
L_\Lambda^\omega(t)&=&\frac{1}{|\Lambda|} \int_{[0,\infty)}{\rm e}^{-t\lambda}\ell_{\Lambda}^{\omega}({\rm d}\lambda)= \frac{1}{|\Lambda|}{\rm Tr}\, P_t^{V^{\omega},\Lambda}
=
 \frac{1}{|\Lambda|}\int_\Lambda p^{V^{\omega},\Lambda}_t(x,x) {\rm d}x
\\
&=&
 \frac{p_t(0,0)}{|\Lambda|}\int_\Lambda \ex_{x,x}^t\left[{\rm e}^{-\int_0^t V^\omega(X_s){\rm d}s}; t < \tau_\Lambda \right] {\rm d}x,
 \end{eqnarray*}
 for any fixed $t>0.$ Denoting by $L$ the Laplace transform of the measure $\ell,$ we have
\begin{equation}\label{eq:el-almost-everywhere}
L(t) = \lim_{\Lambda \nearrow \R^d} \mathbb E^{\mathbb Q}L_\Lambda^\omega(t), \quad t>0.
\end{equation}
Let us note that from  the $\Z^d-$stationarity of the potential we have that for any $\Lambda$ as above,
\begin{equation}\label{eq:IDS-formula}
L(t)=\frac{p_t(0,0)}{|\Lambda|}\int_\Lambda  \mathbb E^{\mathbb Q}\mathbf E_{x,x}^t\left[{\rm e}^{-\int_0^t V^\omega(X_s){\rm d}s}\right]{\rm d}x.
\end{equation}
In particular, for any $M\in\Z_+,$
\begin{equation}\label{eq:IDS-formula_2}
L(t)=\frac{p_t(0,0)}{M^d}\int_{[0,M)^d}  \mathbb E^{\mathbb Q}\mathbf E_{x,x}^t\left[{\rm e}^{-\int_0^t V^\omega(X_s){\rm d}s}\right]{\rm d}x.
\end{equation}

In the next two sections we will determine the rate of decay, as $t\to\infty,$ of   the Laplace transform  $L(t)$ of  the measure $\ell.$

\section{The upper bound for the Laplace transforms}\label{sec:upper}
We start with the upper bounds, as they will determine  the correct rate(s) in the asymptotics. We have more flexibility with lower bounds, thus the  crucial step is to get a correct upper bound.
To put oneself in a proper perspective, let us recall  that for L\'{e}vy operators satisfying {\bf (B)} perturbed by a Poissonian-type potential, the decay of $L(t)$ was of order ${\rm e}^{-Ct^{\frac{d}{d+\alpha}}}.$ We obtained a similar rate for  the Anderson model for the fractional Laplacians, provided the distribution of the random variables $q_{\mathbf i}$ had an atom at zero. However,
when the atom at zero is not present, then our earlier work \cite{bib:KK-KPP-alloy-stable} indicates that an extra multiplicative input
is needed in the decay rate. As Theorems \ref{th:upper-short} and \ref{th:lower} show, this is indeed the case.

\subsection{The rate function and the statement of the upper bound (Theorem \ref{th:upper-short})} \label{sec:rate_der}

We start with the definition of the function $h(t)$ which will appear in  the rate.

Let $\alpha \in(0,2]$ and $C_1>0$ be the scaling exponent and the constant from the assumption {\bf (B)}, and let $\kappa_0$ and $M_0$ be the parameters appearing in the assumptions {\bf (Q)} and {\bf (W)}, respectively. Moreover, recall that by $\mu_2^1$ we have denoted the second eigenvalue of the operator $-\Delta^1$ (see Section \ref{sec:operators_on_tori}). Denote
 \begin{equation}\label{eq:c-zero}
	 D_0=\frac{1}{2}\frac{C_1 (\mu_2^1)^{\frac{\alpha}{2}}\|W\|_1}{\|W\|_1^2+(2M_0)^d\|W\|_2^2} > 0.
	 \end{equation}
As it will be seen below, in fact $D_0$ can be choosen to be an arbitrary constant for which
	 \begin{equation}\label{eq:D-def}
	 \|W\|_1 > \frac{D_0(2M_0)^d\|W\|_2^2}{C_1 (\mu_2^1)^{\frac{\alpha}{2}}-D_0\|W\|_1}>0,
	 \end{equation}
but for more clarity we prefer to keep $D_0$ fixed as in \eqref{eq:c-zero}.

For
 $x>0$ let
\begin{align}\label{eq:Phi}
g(x)=  \log\frac{1}{F_q(D_0/x)} = -\log \mathbb Q[q\leq D_0/x],\quad \mbox{ and } \quad  j(x)= x^{d+\alpha}g({x^\alpha}).
\end{align}
 Due to the assumption {\bf (Q)}  the function $j(x)$ is increasing, and continuous for $x\geq  x_0:= (D_0/\kappa_0)^{1/\alpha} .$ Therefore
  $j^{-1}(t)$ is well-defined for $t\geq t_0:= j(x_0).$ Let $x_t=j^{-1}(t),$ $t\geq t_0.$
Finally, denote
\begin{equation}\label{eq:def-ht}
h(t) = g(x_t^{\alpha}),\quad t\geq t_0.
\end{equation}
For later use, observe that $t$ and $x_t$ are related through the relation
\begin{equation}
\label{eq:t-and-xt}
t= x_{t}^{d+\alpha}\log\left(\frac{1}{F_q(D_0/x_t^\alpha)}\right)= x_t^{d+\alpha}h(t), \quad  t \geq t_0.
\end{equation}
Moreover, the  function $t\mapsto x_t$ is increasing  and since $
\lim_{x\to\infty}j(x)=\infty,$ we have $\lim_{t\to\infty}x_t=\infty$ as well.
  This implies that
 \begin{equation}
\label{eq:lim-h-1}
\lim_{t\to\infty}\frac{h(t)}{t} =0.
 \end{equation}
The limit $\lim_{t \to \infty} h(t)$ always exists and
\begin{align}\label{eq:lim_h}
\lim_{t \to \infty} h(t) = \begin{cases}
\infty \quad \, \ \ \ \ \ \  \text{if} \quad F_q(0) = 0, \\
\log\frac{1}{F_q(0)} \quad \text{if} \quad F_q(0) \in (0, 1).
\end{cases}
\end{align}
In Section \ref{sec:examples} we give examples of such functions $h.$

We are now ready to present the main theorem of this section.
\begin{theorem}\label{th:upper-short} Assume  {\bf (B)}, {\bf (Q)}, and {\bf (W)}.  Let $h$ be given by \eqref{eq:def-ht}.
 Then there exists $C>0$ such that
	\begin{equation}\label{eq:upper-1}
	\limsup_{t\to\infty} \frac{\log L(t)}{t^{\frac{d}{d+\alpha}} (h(t))^{\frac{\alpha}{d+\alpha}}} \leq -C.
	\end{equation}	
 In particular, when the distribution of $q$ has an atom at zero, i.e. $F_q(0) >0$, then
 	\begin{equation}\label{eq:upper-2}
 	\limsup_{t\to\infty}\frac{\log L(t)}{t^{\frac{d}{d+\alpha}}}\leq -C \left(\log\frac{1}{F_q(0)}\right)^{\frac{\alpha}{d+\alpha}} .
 	\end{equation}	
\end{theorem}
The proof of the theorem is split into three  parts which are presented in Sections \ref{sec:trace}, \ref{sec:temple} and
\ref{sec:conclusion} below.

\subsection{Preparatory steps in the proof of Theorem \ref{th:upper-short} - the trace estimate}\label{sec:trace}

	To begin the proof, we proceed as in \cite[Proof of Theorem 4.1, the first page of Section 5.3]{bib:KK-KPP-alloy-stable}.
	As a corollary of \cite[Lemma 5.1]{bib:KK-KPP-alloy-stable}
	we get that for any given $M\in\Z_+$ and any $t>0$ we have the following relation between the exponential functionals of the  subordinate Brownian motion $X = (X_t)_{t \geq 0}$ in $\R^d$ with the un-periodized and periodized potentials  (cf.\ \eqref{eq:alloy}-\eqref{eq:szn-per-of-V}):
	\[\mathbb E^\mathbb Q\left[{\rm e}^{-\int_0^t V^\omega(X_s(w)){\rm d}s}\right] \leq  \mathbb E^\mathbb Q\left[{\rm e}^{-\int_0^t V_M^\omega(X_s(w)){\rm d}s}\right],\]
	and further, invoking Lemma \ref{lm:rotation},
 	\begin{equation}\label{eq:old-argument}
	L(t)\leq \frac{1}{M^d}
	\mathbb E^{\mathbb Q}\int_{\mathcal T_M} p^M_t(x,x)\mathbf E_{x,x}^{M,t}\left[{\rm e}^{-\int_0^t V_M^{\omega}(X_s^M){\rm d}s}\right]{\rm d}x
	\end{equation}
	(recall that $X^M=(X_s^M)_{s\geq 0}$ is the subordinate Brownian motion on the torus $\mathcal T_M,$  $p_t^M(\cdot,\cdot)$ are its transition densities, and $\mathbf E_{x,x}^{M,t}$ - its bridge measures).
	
	The integral at the right-hand side of \eqref{eq:old-argument} is the trace of the  random
	operator
	$P_t^{M,V^{\omega}_M}$ with integral kernel defined in \eqref{eq:FK-torus}.
	Consequently,  for $t>1$,
	\begin{eqnarray*}L(t)&\leq &\frac{1}{M^d} \mathbb E^{\mathbb Q}\mbox{Tr}\,P_t^{M,V^{\omega}_M}
		= \frac{1}{M^d} \mathbb E^{\mathbb Q}\sum_{n=1}^\infty {\rm e}^{-t\lambda_n^{{M, V^{\omega}_M}}}
		\\
		&\leq & \frac{1}{M^d} \mathbb E^{\mathbb Q}\left[{\rm e}^{-(t-1)\lambda_1^{M, V^{\omega}_M}}\mbox{Tr}\,P_1^{M,V^{\omega}_M}\right]
		\\
		&\leq &
		\mathbb E^{\mathbb Q}\left[{\rm e}^{-(t-1)\lambda_1^{M, V^{\omega}_M}}\right]  \frac{1}{M^d}  \int_{\mathcal T_M} p^M_1(x,x){\rm d}x.
	\end{eqnarray*}
From Lemma \ref{lem:regularity}\,(2) there exists a constant $C>0$ independent of $M$ for which $p^M_1(x,x)\leq C,$ $x\in \mathcal T_M,$
	so that we are led to the bound
	\begin{equation}\label{eq:stop}
	L(t)\leq C
	\mathbb E^{\mathbb Q}\left[{\rm e}^{-(t-1)\lambda_1^{M, V^{\omega}_M}}\right], \quad t > 1.
	\end{equation}

	\subsection{Temple's inequality and the lower scaling of the ground state eigenvalue} \label{sec:temple}
	 In this section we find an appropriate lower estimate for the ground state eigenvalue $\lambda_1^{M,V^{\omega}_M}$ of the Schr\"odinger operator $H^{\omega}_M$. We will use the following inequality.
	
	\begin{proposition}[{Temple's inequality, \cite[Theorem XIII.5]{bib:RS}}] \label{prop:temple} Suppose $H$ is a self-adjoint operator  on a Hilbert space with inner product $\langle \cdot,\cdot \rangle$ such that $\lambda_1:=\inf\sigma(H)$ is an isolated eigenvalue and let $\mu\leq \inf(\sigma(H)\setminus \{\lambda_1\}).$  Then for any $\psi\in\mathcal D(H)$ which satisfies
	\begin{equation}\label{eq:temple-condition}
	\langle \psi, H\psi\rangle <\mu\quad\mbox{ and } \quad  \|\psi\|=1
	\end{equation}
	the following bound holds:
	\begin{equation}\label{eq:temple}
	\lambda_1\geq \langle \psi, H\psi\rangle - \frac{\langle H\psi, H\psi\rangle-\langle \psi,H\psi\rangle^2}{\mu-\langle \psi, H\psi\rangle}.
	\end{equation}
\end{proposition}
	
	\bigskip

	For given $M\in\mathbb Z_+, $
	 consider truncated random variables
	\begin{align}\label{eq:trunc_qs}
	 \widetilde{q}_{\mathbf i}=q_{\mathbf i}\wedge \frac{D_0}{M^{\alpha}},\quad \mathbf i\in\mathbb Z^d,
	\end{align}
	 and random Schr\"odinger operators
	 \begin{equation}\label{eq:tilde-H}
	 \widetilde{H}^\omega_M= \Phi(-\Delta^M)+\widetilde V_M^\omega,
 	\end{equation}
	 where $\widetilde {V}_M^\omega$ is the Sznitman-periodization of
	 \[\widetilde{V}^\omega(x)=\sum_{\mathbf i \in \Z^d}\widetilde{q}_{\mathbf i}(\omega) W(x-\mathbf i),\]
	 cf.\ \eqref{eq:alloy}-\eqref{eq:szn-per-of-V}.  We have a lemma.
	
	 \begin{lemma}\label{lem:lambda} Let the assumptions {\bf (B)} and {\bf (W)} hold. Then for any $M\geq M_0$
	 	and any constant $D_0>0$ satisfying \eqref{eq:D-def} we have
	 	\begin{equation}\label{eq:lambda-temple}
	 	\lambda_1^{M, {V}_M^\omega} \geq\lambda_1^{M, \widetilde {V}_M^\omega} \geq \frac{1}{M^d}\left[\int_{\mathcal T_M} \widetilde{
	 		V}_M^{\omega}(x){\rm d}x - \frac{\int_{\mathcal T_M}(\widetilde{V}_M^{\omega}(x))^2{\rm d}x}{\big(C_1 (\mu_2^1)^{\frac{\alpha}{2}}-D_0\|W\|_1\big)\cdot M^{-\alpha}}\right].
	 	\end{equation}
	 \end{lemma}
	 \begin{proof}
	 	Let $M\geq M_0$ be fixed. Since $\widetilde{V}^\omega_M\leq V_M^\omega,$ the leftmost inequality in \eqref{eq:lambda-temple} is clear.
We now apply Temple's inequality  to the operator $
\widetilde{H}^\omega_M$  acting in $L^2(\mathcal T_M)$, defined in \eqref{eq:tilde-H}, and $\mu=\lambda_{2}^M.$
 The spectrum of $\widetilde{H}^\omega_M$ is purely discrete  and it is clear that
\[\mu \leq \lambda_2^{M, \widetilde {V}_M^\omega}= \inf \left(\sigma(\widetilde {H}^\omega_M) \setminus \left\{\lambda_1^{M, \widetilde {V}_M^\omega}\right\}\right).\]
Let  $\psi = \psi_1^M \equiv \frac{1}{M^{d/2}}$. We have $\|\psi\|_2=1$ and,  by \eqref{eq:eigenvalue_subord}-\eqref{eq:eigenvalue_deltaM},
$\Phi(-\Delta^M) \psi=0$. Consequently,
$$
	\langle \psi, \widetilde{H}^\omega_M\psi\rangle =  \langle \psi, \Phi(-\Delta^M)\psi\rangle  + \langle \psi, \widetilde{V}_M^{\omega}\psi\rangle\\
	=\langle \psi, \widetilde{V}_M^{\omega}\psi\rangle= \frac{1}{M^d}\int_{\mathcal T_M} \widetilde{V}_M^{\omega}(x)\,{\rm d}x.
$$	
By the definition of $\widetilde{V}_M^{\omega}$, we have
\begin{align} \label{eq:per_int}
\int_{\mathcal T_M} \widetilde{V}_M^{\omega}(x)\,{\rm d}x & = \int_{\mathcal T_M} \sum_{\mathbf i\in [0,M)^d} \widetilde{q}_{\mathbf i}(\omega)\left(\sum_{\mathbf i'\in\pi_M^{-1}(\mathbf i)} W(x-\mathbf i')\right){\rm d}x\nonumber \\
	&=\sum_{\mathbf i\in [0,M)^d} \widetilde{q}_{\mathbf i}(\omega)\left(\sum_{\mathbf i'\in\pi_M^{-1}(\mathbf i)} \int_{[0,M)^d -\mathbf i'} W(x)\,{\rm d}x\right) = \|W\|_1  \left(\sum_{\mathbf i\in [0,M)^d} \widetilde{q}_{\mathbf i}(\omega)\right). 
\end{align}
Hence, by \eqref{eq:trunc_qs} and \eqref{eq:c-zero},
$$
\langle \psi, \widetilde{H}^\omega_M\psi\rangle \leq \frac{1}{M^d}\,\|W\|_1\cdot M^d\cdot \frac{D_0}{M^\alpha}= \frac{D_0\|W\|_1}{M^\alpha}
< \frac{C_1 (\mu_2^1)^{\frac{\alpha}{2}}}{M^\alpha}.
$$
On the other hand,  from a combination of the lower bound in \eqref{eq:assum-phi-close-to-0} and \eqref{eq:scaling_eig}-\eqref{eq:eigenvalue_subord} it follows  that
$$
\frac{C_1 (\mu_2^1)^{\frac{\alpha}{2}}}{M^\alpha} \leq \Phi(M^{-2} \mu_2^1) = \Phi(\mu_2^M) = \lambda_{2}^M,
$$
and therefore condition \eqref{eq:temple-condition} is satisfied. For the ingredients of \eqref{eq:temple} we have:
\begin{eqnarray*}
	\langle \psi,\widetilde{H}^\omega\psi\rangle &=& \frac{1}{M^d}
	\int_{\mathcal T_M}\widetilde {V}_M^\omega(x)\,{\rm d}x,\\
	\langle \widetilde{H}^\omega\psi,\widetilde{H}^\omega\psi\rangle &=&
\frac{1}{M^d}
\int_{\mathcal T_M}\left(\widetilde {V}_M^\omega(x)\right)^2{\rm d}x,\\
\mu-\langle \psi,\widetilde{H}^\omega\psi\rangle &=& \lambda_2^M- \langle \psi,\widetilde{H}^\omega\psi\rangle
\geq  \frac{C_1 (\mu_2^1)^{\frac{\alpha}{2}}}{M^\alpha}-\frac{D_0\|W\|_1}{M^\alpha} = \frac{C_1 (\mu_2^1)^{\frac{\alpha}{2}}-D_0\|W\|_1}{M^\alpha}.
\end{eqnarray*}
Inserting these inside \eqref{eq:temple} we get
\begin{eqnarray*}
	\lambda_1^{M, {V}_M^\omega}&\geq & 	 \lambda_1^{M, \widetilde {V}_M^\omega}\\
	&\geq& \frac{1}{M^d} \int_{\mathcal T_M}
	\widetilde{V}_M^{\omega}(x)\,{\rm d}x - \frac{\frac{1}{M^d}  \int_{\mathcal T_M}
		\left(\widetilde{V}_M^{\omega}(x)\right)^2\,{\rm d}x -\left(\frac{1}{M^d} \int_{\mathcal T_M}
		\widetilde{V}_M^{\omega}(x)\,{\rm d}x\right)^2}{(C_1 (\mu_2^1)^{\frac{\alpha}{2}}-D_0\|W\|_1) M^{-\alpha}}\\
	&\geq& \frac{1}{M^d}\left[  \int_{\mathcal T_M}
	\widetilde{V}_M^{\omega}(x)\,{\rm d}x-\frac{1}{(C_1 (\mu_2^1)^{\frac{\alpha}{2}}-D_0\|W\|_1)M^{-\alpha}}  \int_{\mathcal T_M}
	\left(\widetilde{V}_M^{\omega}(x)\right)^2\,{\rm d}x\right],
\end{eqnarray*}
which is the desired statement.	
\end{proof}
This lemma will be useful when there is a lot of randomness in the picture, namely when the random variables $q_{\mathbf i}$ are bigger than $D_0/ M^{\alpha}$ on a substantial part of sites in $\mathcal T_M.$ To quantify this behavior, fix $\delta\in(0,1)$ (its actual value will be decided later) and consider the
set
\begin{equation}\label{eq:a-delta-def}
\mathcal A_{M,\delta}=\left\{\omega:\#\left\{\mathbf i\in [0,M)^d: q_{\mathbf i}(\omega)  >  \frac{D_0}{M^\alpha}\right\}\geq  \delta M^d  \right\}.
\end{equation}
We have the following estimate.
\begin{lemma}\label{lem:lambda-on-a-delta}  Let the assumptions {\bf (B)} and {\bf (W)} hold and  let $\delta>0$ be fixed.
	Suppose $\omega\in\mathcal A_{M,\delta}.$ Then for any  $M \geq M_0$
	\begin{equation}\label{eq:lambda-on-a-delta}
	\lambda_1^{M, {V}_M^\omega} \geq D_0\cdot\delta \left[\|W\|_1-\frac{(2M_0)^dD_0\|W\|_2^2}{C_1 (\mu_2^1)^{\frac{\alpha}{2}}-D_0\|W\|_1}\right]\cdot\frac{1}{M^\alpha}.
	\end{equation}
\end{lemma}
\begin{proof}
 We have already shown in \eqref{eq:per_int} that $\int_{\mathcal T_M}
	\widetilde{V}_M^{\omega}(x)\,{\rm d}x = \|W\|_1  \left(\sum_{\mathbf i\in [0,M)^d} \widetilde{q}_{\mathbf i}(\omega)\right)$.
	Under present assumptions, we will also find a nice etimate on $\int_{\mathcal T_M}
	\left(\widetilde{V}_M^{\omega}(x)\right)^2\,{\rm d}x$
	and then we will apply Lemma \ref{lem:lambda}.
Observe that because of the assumption $\mbox{supp}\,W\subset  [-M_0,M_0]^d,$ in the sum defining $\widetilde{V}_M^\omega,$
 \[
\widetilde{V}_M^{\omega}(x) = \sum_{\mathbf i\in[0,M)^d} \widetilde{q}_{\mathbf i}(\omega)\sum_{\mathbf i'\in\pi_M^{-1}(\mathbf i)}W(x-\mathbf i')\]	
there are at most $(2M_0)^d$ nonzero terms  and for every fixed $\mathbf i\in[0,M)^d$ the range of the summation $\pi_M^{-1}(\mathbf i)$ in the inner sum contains at most one element.
Consequently,
\[
\left(\widetilde{V}_M^{\omega}(x)\right)^2 \leq (2M_0)^d \sum_{\mathbf i\in[0,M)^d} \widetilde{q}_{\mathbf i}(\omega)^2\sum_{\mathbf i'\in\pi_M^{-1}(\mathbf i)}W(x-\mathbf i')^2,\]
and further, as in the proof of \eqref{eq:per_int},
\[\int_{\mathcal T_M}
\left(\widetilde{V}_M^{\omega}(x)\right)^2\,{\rm d}x\leq (2M_0)^d\sum_{\mathbf i\in[0,M)^d} \widetilde{q}_{\mathbf i}^2(\omega) \|W\|_2^2.\]
Inserting these estimates inside \eqref{eq:lambda-temple} we obtain:
\begin{eqnarray*}
	\lambda_1^{M, {V}_M^\omega}(\mathcal T_M) &\geq& \frac{1}{M^d}\left[
	\|W\|_1\sum_{\mathbf i\in [0,M)^d}\widetilde{q}_{\mathbf i}(\omega) -
	\frac{(2M_0)^d}{(C_1 (\mu_2^1)^{\frac{\alpha}{2}}-D_0\|W\|_1)M^{-\alpha}}\|W\|_2^2\sum_{\mathbf i\in[0,M)^d} \widetilde{q}_{\mathbf i}^2(\omega)\right]\\
	&=& \frac{1}{M^d}\sum_{\mathbf i\in[0,M)^d}\widetilde{q}_{\mathbf i}(\omega)\left[ \|W\|_1-\frac{(2M_0)^d\|W\|_2^2 \sum_{\mathbf i\in[0,M)^d} \widetilde{q}_{\mathbf i}^2(\omega)}{(C_1 (\mu_2^1)^{\frac{\alpha}{2}}-D_0\|W\|_1)M^{-\alpha}} \left(\sum_{\mathbf i\in[0,M)^d} \widetilde{q}_{\mathbf i}(\omega)\right)^{-1}\right].
\end{eqnarray*}
	
	Now: in the sum $\sum_{\mathbf i\in[0,M)^d}\widetilde{q}_{\mathbf i}(\omega)$ we keep only those $\mathbf i$'s for which $q_{\mathbf i} >  \frac{D_0}{M^\alpha}.$ Because of the assumption $\omega\in\mathcal A_{M,\delta},$ this leads to
	\[\sum_{\mathbf i\in[0,M)^d}\widetilde{q}_{\mathbf i}(\omega) \geq \frac{D_0}{M^\alpha}\cdot \delta M^d,\]
	and for $\sum_{\mathbf i\in[0,M)^d}\widetilde{q}^2_{\mathbf i}(\omega)$ we write
	\[\sum_{\mathbf i\in[0,M)^d}\widetilde{q}^2_{\mathbf i}(\omega)\leq
	\frac{D_0}{M^\alpha}\sum_{\mathbf i\in[0,M)^d}\widetilde{q}_{\mathbf i}(\omega).\]
	Finally,
	\begin{eqnarray*}
		\lambda_1^{M,{V}_M^\omega} &\geq& \frac{D_0\delta}{M^\alpha}\left[\|W\|_1-\frac{(2M_0)^dD_0\|W\|_2^2}{C_1 (\mu_2^1)^{\frac{\alpha}{2}}-D_0\|W\|_1}\right].
	\end{eqnarray*}
\end{proof}

\subsection{Conclusion of the proof of Theorem \ref{th:upper-short} - derivation of the rate function.} \label{sec:conclusion}

 To conclude the proof  we continue with the estimate of $L(t)$ from \eqref{eq:stop}, splitting the $\mathbb Q-$integration into two parts: over $\mathcal A_{M,\delta}$ and over its complement.

For the integral over $\mathcal A_{M,\delta}$ we have the estimate from Lemma \ref{lem:lambda-on-a-delta}, and the integral over $\mathcal A_{M,\delta}^c$ is not bigger than  $\mathbb Q[\mathcal A_{M,\delta}^c],$ whose probability can be estimated by the following Bernstein-type inequality  on the binomial distribution.  Its proof is an exercise from elementary probability, but we give here a short proof for the reader's convenience.

 \begin{lemma}\label{lem:bernoulli} Let $(\Omega, \mathcal F, \mathbb P)$ be a given probability space and let $S_n:
 	\Omega\to\mathbb R$ be a random variable with the binomial distribution $B(n,p),$ $n\geq 1, p\in(0,1).$ Then, for any $p,\gamma\in(0,1)$   such that $\gamma>p,$
 	\begin{equation}\label{eq:bern-exp}
 	\mathbb P[S_n  \geq  \gamma n]\leq\left(\left(\frac{1-p}{1-\gamma}\right)^{1-\gamma}\left(\frac{p}{\gamma}
 	\right)^\gamma\right)^n.
 	\end{equation}
 \end{lemma}
 \begin{proof}
 	For any given $t>0,$ we have the following estimate, deduced from the Markov inequality:
 	\begin{eqnarray*}
 		\mathbb P[S_n \geq \gamma n]&  =  & \mathbb P\left[{\rm e}^{tS_n} \geq {\rm e}^{t\gamma n}\right]
 		\leq \frac{\mathbb E{\rm e}^{tS_n}}{{\rm e}^{t\gamma n}}
 		= \left(\frac{p{\rm e}^t+1-p}{{\rm e}^{t\gamma}}\right)^n.
 	\end{eqnarray*}
 	The minimal value of the right-hand side is taken at $t_\gamma= \log\left(\frac{\gamma}{p}\,\frac{1-p}{1-\gamma}\right),$ which is positive for $\gamma>p.$ This value is equal to
 	 $\left(\left(\frac{1-p}{1-\gamma}\right)^{1-\gamma}\left(\frac{p}{\gamma}
 	\right)^\gamma\right)^n,$ as claimed.
 \end{proof}

\begin{proof}[Proof of Theorem \ref{th:upper-short}] For an arbitrary $\delta \in (0,1)$ we can write
\begin{equation} \label{eq:stop1}
L(t)\leq C\mathbb E^{\mathbb Q}\left[{\rm e}^{-(t-1)\lambda_1^{M, {V}_M^\omega}}; \mathcal A_{M,\delta}\right]+ C\mathbb Q[\mathcal A _{M,\delta}^c], \quad t > 1, \ \ M \geq M_0.
\end{equation}
To make use of Lemma \ref{lem:bernoulli} observe
\begin{eqnarray*}
	\mathcal A_{M,\delta}^c&=&\left\{\omega: \#\left\{\mathbf i\in[0,M)^d: q_{\mathbf i}(\omega)>\frac{D_0}{M^\alpha}\right\}< \delta M^d \right\}\\
	&=&
	\left\{\omega: \#\left\{\mathbf i\in[0,M)^d: q_{\mathbf i}(\omega)\leq\frac{D_0}{M^\alpha}\right\} \geq   (1-\delta) M^d \right\}.
\end{eqnarray*}
The events $A_{\mathbf i}=\{q_{\mathbf i}\leq\frac{D_0}{M^\alpha}\}$ are independent and have common probability $ p_M=F_q(\frac{D_0}{M^\alpha}).$ Therefore we
use the lemma with $n=M^d,$  $ p = p_M$ as above, and $\gamma=(1-\delta).$
We only need to make sure that  $(1-\delta)>p_M,$ i.e. $\delta < 1-p_M.$
As eventually we will
let $M\to\infty$ and the distribution of the random variable $q$ is not concentrated at $0,$ this will not be a problem.
 It then follows that for every $M\geq M_0$ and $\delta < 1-p_M$ it holds
\begin{align*}
\mathbb Q[\mathcal A_{M,\delta}^c]\leq  \left[\left(\frac{1-p_M}{\delta}\right)^{\delta}\left(\frac{p_M}{1-\delta}\right)^{1-\delta}\right]^{M^d}
\leq\left[\frac{1}{1-\delta}\left(\frac{1}{\delta}\right)^{\frac{\delta}{1-\delta}}\cdot p_M \right]^{(1-\delta)M^d}.
\end{align*}
Since $p_M = F_q(\frac{D_0}{M^\alpha}) \to F_q(0) \in [0,1)$ as $M \to \infty$ and $\frac{1}{1-\delta}\left(\frac{1}{\delta}\right)^{\frac{\delta}{1-\delta}} \searrow 1$ as $\delta \searrow 0$, we can easily find $M_1 \geq M_0$ and $\delta_0 < 1-p_{M_1}$ such that for every $M \geq M_1$
$$
\frac{1}{1-\delta_0}\left(\frac{1}{\delta_0}\right)^{\frac{\delta_0}{1-\delta_0}}\cdot \sqrt{p_M} \leq 1
$$
(in particular, $\delta_0 < 1-p_{M}$, for $M \geq M_1$ as $p_M$ is nonincreasing in $M$). It gives
\begin{align}\label{eq:q-of-ac}
\mathbb Q[\mathcal A_{M,\delta_0}^c] \leq {p_M}^{(1-\delta_0)M^d/2} = {\rm e}^{-\frac{1-\delta_0}{2} \, M^d \, \log (1/p_M)}, \quad M \geq M_1.
\end{align}

Denote $c_1= \frac{D_0\delta_0}{2}\left[\|W\|_1-\frac{(2M_0)^dD_0\|W\|_2^2}{(C_1 (\mu_2^1)^{\frac{\alpha}{2}}-D_0\|W\|_1)}\right]$  and $c_2 = (1-\delta_0)/2$. We now insert the bounds \eqref{eq:lambda-on-a-delta} and \eqref{eq:q-of-ac} inside \eqref{eq:stop1} and obtain that there exist $t_0>1$ such that for every $t \geq t_0$ and $M \geq M_1$ we have
\begin{eqnarray}
	L(t) &\leq & C\left({\rm e}^{-\frac{2c_1(t-1)}{M^\alpha}} +{\rm e}^{-c_2 M^d \log\frac{1}{F_q(D_0/M^\alpha)}}\right) \nonumber\\
	&\leq& C\left({\rm e}^{- \frac{c_1t}{(M-1)^\alpha}} +{\rm e}^{-c_2M^d\log\frac{1}{F_q(D_0/M^\alpha)}}\right) \leq C\left({\rm e}^{- \frac{c_3t}{(M-1)^\alpha}} +{\rm e}^{-c_3 M^d\log\frac{1}{F_q(D_0/M^\alpha)}}\right) \label{eq:el-t},
\end{eqnarray}
with $c_3=\min(c_1,c_2).$

So far, the bound we obtained was valid for any $M>M_1.$ We will now make $M$ depend on $t,$ in such a manner that $M\to\infty$ when $t\to\infty.$
We will use the function $ j(x)=x^{d+\alpha}\log \frac{1}{F_q(D_0/x^\alpha)}$  from \eqref{eq:Phi}. For $t\geq t_0$ (we may increase $t_0$ if necessary), the function
$x_t:= j^{-1}(t)$ is well defined and obeys \eqref{eq:t-and-xt}.
Let us take
\begin{equation}\label{eq:M-def}
M=\lfloor x_t\rfloor+1, \quad t\geq t_0,
\end{equation}
i.e. $M$ is the unique integer satisfying  $x_{t}-1<M-1\leq x_t.$ Clearly, there is $t_1\geq t_0$ such that for $t\geq t_1$ one has $M\geq M_1.$
Consequently, by \eqref{eq:t-and-xt},
\[(M-1)^\alpha\leq x_t^\alpha= \left(\frac{t}{\log\frac{1}{F_q(D_0/x_t^\alpha)}}\right)^{\frac{\alpha}{d+\alpha}} \]
	so that\[\frac{t}{(M-1)^\alpha} \geq t^{\frac{d}{d+\alpha}}\left(\log\frac{1}{F_q(D_0/x_t^\alpha)}\right)^{\frac{\alpha}{d+\alpha}},\quad t\geq t_1.\]
Next, as the function $x\mapsto x^d\log\frac{1}{F_q(D_0/x^\alpha)}$ is increasing and $x_t<M,$ for the other exponent in \eqref{eq:el-t} we have
\[M^d\log\frac{1}{F_q(D_0/M^\alpha)}\geq x_t^d\log\frac{1}{F_q(D_0/x_t^\alpha)}= t^{\frac{d}{d+\alpha}}\left(\log\frac{1}{F_q(D_0/x_t^\alpha)}\right)^{\frac{\alpha}{d+\alpha}}.\]	 
Consequently,
\[L(t)\leq 2C {\rm e}^{-D t^{\frac{d}{d+\alpha}}\left(\log\frac{1}{F_q(D_0/x_t^\alpha)}\right)^{\frac{\alpha}{d+\alpha}}}=2C{\rm e}^{-D t^{\frac{d}{d+\alpha}} (h(t))^{\frac{\alpha}{d+\alpha}}},\]
which yields \eqref{eq:upper-1}.  The second assertion \eqref{eq:upper-2} is an easy consequence of  \eqref{eq:upper-1} and  \eqref{eq:lim_h}.
\end{proof}

\bigskip

\section{The lower bound for the Laplace transforms}\label{sec:lower}
The matching lower bound will be obtained by restricting the integration in the integrals leading to $L(t)$ (see \eqref{eq:IDS-formula_2}) to a smaller set, on which we will be able to control the expressions from below, and whose probability will be manageable.  On this set we replace our random potential $V^{\omega}$ with  deterministic potentials
$$
V_{\kappa}(x) := \kappa \sum_{\mathbf i\in[-M,2M)^d} W(x- \mathbf i), \quad \text{where \ $\kappa >0$ \ is some specially chosen parameter,}
$$
and then we estimate from the above the ground state eigenvalues of the Schr\"odinger operators
$$
\Phi(-\Delta) + V_{\kappa}
$$
constrained to large boxes in $\R^d$. Note that in this section we work with the operators $\Phi(-\Delta)$ and the subordinate Brownian motions in $\R^d$ and its sub-domains only, and we need not consider the operators and the processes on tori. We first recall necessary notation (cf. Section \ref{sec:dirichlet}).

If $V \in \mathcal K_{\loc}$ is a nonnegative potential and $\Lambda$ is a bounded domain in $\R^d$, then by $H_{\Lambda}$ we denote the Schr\"odinger operator
$H = \Phi(-\Delta) + V$ constrained to $\Lambda$ (i.e. with  Dirichlet conditions on $\Lambda^c$ in the non-local case and on $\partial \Lambda$ in the local case)
and by $P_t^{V,\Lambda} = e^{-t H_{\Lambda}}$ the operators of its evolution semigroup. The ground state eigenvalue $\lambda_1^{V}(\Lambda)$ of $H_{\Lambda}$ can be represented through the variational formula
\begin{align} \label{eq:dirichlet_variat}
	\lambda_1^V(\Lambda)=\inf \left\{\mathcal E(\varphi,\varphi) + \int_{\Lambda} V(x) \varphi^2(x) dx: \varphi \in L^2(\Lambda), \|\varphi\|_2=1\right\}.
\end{align}
This infimum is achieved for $\varphi_1^{V,\Lambda}$, the ground state eigenfunction of $H_{\Lambda}$. Below we also consider the case when $V \equiv 0$ for which we use simpler notation: $P_t^{\Lambda}$, $\lambda_1(\Lambda)$ and $\varphi_1^{\Lambda}$.

We start with an auxiliary lemma.
\begin{lemma}\label{lem:lambda-f}
	Let $V:\mathbb R^d\to\mathbb R_+$ be a potential that belongs
	to $\mathcal K_{\rm loc}\cap L^1(\mathbb R^d).$ Then for any domain $\Lambda \subset \R^d$ we have
	$$
	\lambda_1^V(\Lambda)\leq \lambda_1(\Lambda) + {\rm e} \, p_{s}(0) \, \|V\|_1, \quad \text{with} \ \  s:= \frac{1}{\lambda_1(\Lambda)}.
	$$
\end{lemma}

\begin{proof}
	 Choosing $s=\frac{1}{\lambda_1(\Lambda)}$ in the eigenequation $P_{s/2}^{\Lambda}\varphi_1^{\Lambda}={\rm e}^{-(s/2) \lambda_1(\Lambda)}\varphi_1^{\Lambda},$ we get
	\begin{eqnarray*}
		\varphi_1^{\Lambda}(x) &= & \sqrt{{\rm e}} \int_{\Lambda}\varphi_1^{\Lambda}(y) p^{\Lambda}_{s/2}(x,y){\rm d}y
		 \leq  \sqrt{{\rm e}}\left(\int_{\Lambda}\left(\varphi_1^{\Lambda}(y)\right)^{2}{\rm d}y\right)^{1/2}
		\left(\int_{\Lambda} (p^{\Lambda}_{s/2}(x,y))^2{\rm d}y\right)^{1/2}\\
		&\leq& \sqrt{{\rm e}} \, \|\varphi_1^{\Lambda}\|_2 \, (p_s(x,x))^{1/2} = \sqrt{{\rm e} \, p_s(0)} .
	\end{eqnarray*}
	By \eqref{eq:dirichlet_variat} we then obtain
	\begin{eqnarray}\label{eq:lambda-f-1}
	\lambda_1^V(\Lambda)\leq \mathcal E(\varphi_1^{\Lambda},\varphi_1^{\Lambda}) +\int_{\Lambda} V(x)\big(\varphi_1^{\Lambda}(x)\big)^2{\rm d}x
	\leq \lambda_1(\Lambda) +{\rm e} \, p_s(0) \, \|V\|_1.
	\end{eqnarray}
\end{proof}

The following is the main theorem of this section.
\begin{theorem}\label{th:lower}
Assume {\bf (B)},  {\bf (Q)},  and {\bf (W)}. Let $h$ be given by \eqref{eq:def-ht}.
 Then there exists $C>0$ such that
	\begin{equation}\label{eq:lower-1}
	\liminf_{t\to\infty} \frac{\log L(t)}{t^{\frac{d}{d+\alpha}} (h(t))^{\frac{\alpha}{d+\alpha}}} \geq -C.
	\end{equation}	
 In particular, when the distribution of $q$ has an atom at zero, i.e. $F_q(0) >0$, then
 	\begin{equation}\label{eq:lower-2}
 	\liminf_{t\to\infty}\frac{\log L(t)}{t^{\frac{d}{d+\alpha}}} \geq -C \left(\log\frac{1}{F_q(0)}\right)^{\frac{\alpha}{d+\alpha}} .
 	\end{equation}	
\end{theorem}

\begin{proof}
	 Let $\Lambda_M:= [0,M)^d$, $M \in \Z_+$.
	  Recall that by \eqref{eq:IDS-formula_2}, for any $t > 0$, we have
	\begin{eqnarray}\label{eq:lt}
	L(t)&=&  \frac{p_t(0,0)}{M^d}\int_{\Lambda_M}\mathbf E_{x,x}^t\mathbb E^{\mathbb Q}\left[{\rm e}^{-\int_0^t V^\omega(X_s)\,{\rm d}s}\right]{\rm d}x
	\nonumber\\
	&=& \frac{1}{M^d}\int_{\Lambda_M} p_t(x,x)\mathbf E_{x,x}^t\mathbb E^{\mathbb Q}\left[{\rm e}^{-\int_0^t V^\omega(X_s)\,{\rm d}s}\right]{\rm d}x, \quad M\in\Z_+.
	\end{eqnarray}
	For given  $M\geq M_0$ (recall that $M_0$ comes from the Assumption {\bf (W)})  and $\kappa> 0$ let
	\[
	\mathcal A_{\kappa}^M=\{\omega:\,\forall \, \mathbf i\in  [-M, 2M)^d  \mbox{ we have } q_{\mathbf i}(\omega)\leq\kappa\}\]
	and restrict the inner expectations in  \eqref{eq:lt} to the set
	$\mathcal A_{\kappa}^M\cap\{t< \tau_{ \Lambda_M} \}.$
	For later use, observe that
	\begin{equation}\label{eq:est-qa}
	\mathbb Q[\mathcal A_{\kappa}^M]=(F_q(\kappa))^{ (3M)^d}={\rm e}^{-(3M)^d\log \frac{1}{F_q(\kappa)}}.
	\end{equation}
On the set $\{t< \tau_{ \Lambda_M}\}$ we have
	\[V^\omega(X_s)= \sum_{\mathbf i\in[-M,2M)^d}q_{\mathbf i}(\omega)W(X_s-\mathbf i),\quad s\leq t.\]
	This is due to the fact that when   $x- y \notin [-M,M]^d,$  then $W(x-y)=0.$
	Next, for $\omega\in\mathcal A_{\kappa}^M,$ all the $q_{\mathbf i}$'s in the first sum above are not bigger than $\kappa.$
	It follows that
	\begin{eqnarray}\label{eq:low-2}
	\mathbb E^{\mathbb Q} \left[{\rm e}^{-\int_0^tV^{\omega}(X_s)\,{\rm d}s};\mathcal A_{\kappa}^M\right] &\geq &
	\mathbb E^{\mathbb Q} \left[{\rm e}^{-\int_0^t\kappa\sum_{i\in[-M,2M)^d} W(X_s-\mathbf i){\rm d}s};\mathcal A_{\kappa}^M\right]\nonumber\\
	&\geq&{\rm e}^{-\int_0^t\kappa\sum_{i\in[-M,2M)^d} W(X_s-\mathbf i){\rm d}s}{\mathbb Q}[\mathcal A_{\kappa}^M].
	\end{eqnarray}
	
\noindent	We see that for  $M\geq M_0$
	\begin{eqnarray*}
		L(t)&\geq & \frac{1}{M^d} \mathbb Q[\mathcal A_{\kappa}^M] \int_{ \Lambda_M} p(t,x,x) \mathbf E_{x,x}^t\left[
		{\rm e}^{-\kappa\int_0^t \sum_{\mathbf i\in [-M,2M)^d}W(X_s-\mathbf i){\rm d}s}; t< \tau_{ \Lambda_M}\right]{\rm d}x.
	\end{eqnarray*}
	In the integral over $ \Lambda_M $ we recognize the trace of the operator $P_t^{ V_{\kappa}, \Lambda_M }$ (cf. \eqref{def:sem-dir-kernel-bridge}) on $L^2( \Lambda_M ),$ corresponding to the potential
	\[ V_{\kappa}(x)  = \kappa\sum_{{\mathbf i\in [-M,2M)^d}} W(x-\mathbf i).\]
	
Therefore this integral is not bigger than the principal eigenvalue of the operator  $P_t^{V_{\kappa}, \Lambda_M},$  which in turn can be estimated by Lemma \ref{lem:lambda-f}:
	\begin{equation}\label{eq:tr}
	\mbox{Tr}\, P_t^{V_{\kappa}, \Lambda_M} \geq {\rm e}^{-t \lambda_1^{ V_{\kappa}}( \Lambda_M)} \geq
	 {\rm e}^{-t(\lambda_1(\Lambda_M)+ \, {\rm e} \, p_{s}(0)\|V_{\kappa}\|_1)}, \quad \text{with} \ \ s = \frac{1}{\lambda_1(\Lambda_M)}.
	\end{equation}
	It remains to estimate  the $L^1-$norm of $ V_{\kappa}.$ We have:
	\begin{eqnarray}\label{eq:norm-of-f}
	\|V_{\kappa}\|_1 & = & \kappa\int_{ \R^d} \sum_{\mathbf i\in[-M, 2M)^d} W(x-\mathbf i) \,{\rm d}x =\kappa \sum_{\mathbf i\in[-M, 2M)^d} \int_{ \R^d}W(x){\rm d}x\nonumber\\
	 &  = & \kappa (3M)^d  \|W\|_1.
	\end{eqnarray}	
From the estimates  \eqref{eq:est-qa}, \eqref{eq:tr}, and \eqref{eq:norm-of-f} we then obtain that  for $M\geq M_0$ and $t>0$
\[L(t)\geq \frac{1}{M^d}{\rm e}^{-(3M)^d\log \frac{1}{F_q(\kappa)}} {\rm e}^{-t(\lambda_1(\Lambda_M)+ \, {\rm e} \, \kappa \, p_{s}(0) (3M)^d\|W\|_1)}, \quad \text{with} \ \ s = \frac{1}{\lambda_1(\Lambda_M)}. \]
 Furthermore, it follows from \cite[Theorem 3.4]{bib:CS} that
\[\lambda_1(\Lambda_M) \leq \Phi(\mu_1(\Lambda_M)),\]
where $\mu_1(\Lambda_M)$ is the ground state eigenvalue of the Laplace operator $-\Delta$ on $\Lambda_M$ with Dirichlet boundary conditions.
Since $\mu_1(\Lambda_M) = M^{-2} \mu_1(\Lambda_1)$,  by the upper bound in the assumption {\bf (B)} we obtain that there exist $M_1 \geq M_0$ and a constant $c_1>0$ such that
\[\lambda_1(\Lambda_M)\leq \frac{c_1}{M^\alpha},\quad M \geq M_1.\]
 In particular, by \eqref{eq:diag_on_pt}, there is a constant $c_2>0$, for which we get
$$
p_s(0)\big|_{s = \frac{1}{\lambda_1(\Lambda_M)}} \leq \frac{c_2}{M^d}, \quad \quad M \geq M_1.
$$
Consequently, by choosing $\kappa=\frac{D_0}{M^\alpha},$ where $D_0$ comes from \eqref{eq:c-zero} (same as in the proof of the upper bound), we obtain
\[L(t)\geq \frac{1}{M^d} {\rm e}^{-c_3\left(\frac{t}{(M+1)^\alpha}+M^d \left(\frac{M}{M+1}\right)^{\alpha} \log \frac{1}{F_q(D_0/M^\alpha)}\right)},\quad M\geq M_1,\]
for some constant $c_3>0.$
The exponent can be written as
\begin{equation}\label{eq:expo}
-\,\frac{c_3}{(M+1)^\alpha}\left(t+ M^{d+\alpha}\log \frac{1}{F_q(D_0/M^\alpha)}\right).
\end{equation}
Again, assume $t\geq t_0,$ let  $x_t=j^{-1}(t)$ (for a definition of the function $j$ and $t_0$ see the formula \eqref{eq:Phi} and the two sentences following it), and choose $M=\lfloor x_t\rfloor$. It is the unique integer for which
\[M\leq x_t  < M+1,\quad {\rm i.e.} \quad j(M)\leq t<j(M+1).\]
Condition  $j(M)\leq t$  reads $M^{d+\alpha}\log\frac{1}{F_q(D_0/M^\alpha)}\leq t.$
Moreover,  by \eqref{eq:t-and-xt},
\[\frac{t}{(M+1)^\alpha}\leq\frac{t}{x_t^\alpha} = t^{\frac{d}{d+\alpha}}\left(\log\frac{1}{F_q(D_0/x_t^\alpha)}\right)^{\frac{\alpha}{d+\alpha}}=t^{\frac{d}{d+\alpha}}(h(t))^{\frac{\alpha}{d+\alpha}}.\]
Consequently, there is a number $t_1\geq t_0$ such that for $t\geq t_1$
\[L(t)\geq \frac{1}{M^d} {\rm e}^{-2c_3 t^{\frac{d}{d+\alpha}}(h(t))^{\frac{\alpha}{d+\alpha}}},\]
with $M$ chosen as above, i.e.
\begin{align}\label{eq:lower_final}
\frac{\log L(t)}{t^{\frac{d}{d+\alpha}}(h(t))^{\frac{\alpha}{d+\alpha}}} \geq -2c_3 - d \frac{\log \lfloor x_t\rfloor}{t^{\frac{d}{d+\alpha}}(h(t))^{\frac{\alpha}{d+\alpha}}}.
\end{align}
To conclude, we only need to verify that
\[
\lim_{t\to\infty} \frac{\log \lfloor x_t\rfloor}{t^{\frac{d}{d+\alpha}}(h(t))^{\frac{\alpha}{d+\alpha}}}=0.\]
This is clear as,  by \eqref{eq:t-and-xt},
$t=x_t^{d+\alpha}h(t),$ and further
$(d+\alpha)\log x_t +\log h(t)=\log t, $ so that
\[\frac{\log x_t}{t^{\frac{d}{d+\alpha}}(h(t))^{\frac{\alpha}{d+\alpha}}}=\frac{1}{(d+\alpha)} \frac{\log t- \log h(t)}{t^{\frac{d}{d+\alpha}}(h(t))^{\frac{\alpha}{d+\alpha}}}\to 0,  \quad \text{as} \ \ t \to \infty,\]
because  $\lim_{t \to \infty} h(t)$ always exists and is strictly positive (possibly infinite, see \eqref{eq:lim_h}). Denoting $C=2c_3$  and taking $\liminf$ in \eqref{eq:lower_final}, we obtain \eqref{eq:lower-1}.
The second assertion \eqref{eq:lower-2} is again a direct consequence of \eqref{eq:lim_h}.
\end{proof}

\section{Tauberian theorems and the asymptotics of the IDS}\label{sec:tauber}

We will now transform the estimates for the Laplace transform $ L(t) $ of the IDS obtained in Sections \ref{sec:upper} and \ref{sec:lower} into statements concerning the IDS itself. When  $\log L(t) \asymp -t^{\gamma}$ as $t \to \infty$, with $\gamma \in (0,1)$,
then one just uses the exponential Tauberian theorem \cite[Theorem 2.1]{bib:F},  to get $\log \ell(\lambda)  \asymp -\lambda^{-\gamma/(1-\gamma)},$ $\lambda\searrow 0,$  as it was done previously in
\cite{bib:Nak,bib:Oku,bib:Szn1,bib:KPP1,bib:KK-KPP2,bib:KK-KPP-alloy-stable}. However,  the rate we identified in Theorems \ref{th:upper-short}, \ref{th:lower} is more general (a correction term is present) and the Tauberian theorems existing in the literature are not sufficient to deal with it. Therefore we first need to state and prove a version of exponential Tauberian theorem which can be applied in our situation.

\subsection{Tauberian theorem}

The setting is as follows. Let $\rho({\rm d}x)$ be a $\sigma-$finite Borel measure on $[0,\infty)$ and let $L(t):= \int_{[0,\infty)} e^{-tx} \rho({\rm d}x)$  be its Laplace transform.  We assume that $L(t) < \infty$ for every $t>0$.
We will use the same letter  $\rho$ for the cumulative distribution function of the measure $\rho,$  i.e. $\rho(x) = \rho([0,x])$, $x \geq 0$. Moreover, let $ g:(0,\infty)\to (0,\infty) $ be a nondecreasing function, continuous  on $[x_0, \infty)$, $x_0\geq 0$. Let $\alpha,d>0$ be two given numbers.
For $t\geq t_0:=x_0^{d+\alpha}g(x_0^\alpha)$ there is a unique number $x_t$ such that  $t=x_t^{d+\alpha}g(x_t^\alpha).$  Since $x^{d+\alpha}g(x^\alpha) \to \infty$ as $x \to \infty$, we also have $x_t \to \infty$ as $t \to \infty$. Finally, let $ h(t)=g(x_t^{\alpha}),$ $t\geq t_0.$  Clearly, $\lim_{t \to \infty} h(t)$ exists and $\lim_{t \to \infty} h(t) \in (0,\infty]$. In particular, $t^{\frac{d}{d+\alpha}} (h(t))^{\frac{\alpha}{d+\alpha}} \to \infty$ as $t \to \infty$.

\begin{theorem}\label{th:tauberian} Using the notation introduced above we have the following.
	
(i) If
\begin{equation}\label{eq:taub-lower-assump}
\liminf_{t\to\infty}\frac{\log L(t)}{t^{\frac{d}{d+\alpha}} (h(t))^{\frac{\alpha}{d+\alpha}}} \geq -A_1,
\end{equation}
with certain constant $A_1\in(0,\infty),$
then for any $B_1>A_1$ we have
\begin{equation}\label{eq:taub-lower}
\liminf_{x\to 0^+} \frac{x^{d/\alpha}}{g(B_1/x)}\,\log \rho(x)\geq -A_1  B_1^{d/\alpha}.
\end{equation}

\smallskip

(ii) If
\begin{equation}\label{eq:taub-upper-assump}
\limsup_{t\to\infty}\frac{\log L(t)}{t^{\frac{d}{d+\alpha}} (h(t))^{\frac{\alpha}{d+\alpha}}} \leq -A_2,
\end{equation}
with certain constant $A_2\in(0,\infty),$
then for any $B_2<A_2$
\begin{equation}\label{eq:taub-upper}
\limsup_{x\to 0^+} \frac{x^{d/\alpha}}{g(B_2/x)}\,\log \rho(x)\leq -(A_2-B_2) B_2^{d/\alpha}.
\end{equation}

\end{theorem}

\begin{proof}
(i)  Assume that \eqref{eq:taub-lower-assump} holds.
	To shorten the notation, denote  $\gamma=\frac{d}{d+\alpha}.$ Then the rate in the denominator of \eqref{eq:taub-lower-assump} (and of \eqref{eq:taub-upper-assump}) is
	equal to $t^\gamma (h(t))^{1-\gamma}.$
	Let
	\[\widetilde L(t):=\int _0^\infty  {\rm e}^{-tx}  \rho(x){\rm d}x =\frac{1}{t}\,L(t)\]
	(the last identity is obtained via integration by parts). It follows that \eqref{eq:taub-lower-assump} is satisfied for $\widetilde L(t) $ as well.
	Let $\epsilon >0$ be given. Then there is $t_\epsilon>0$ such that for $t>t_\epsilon$ one has
	\begin{equation}\label{eq:el-tilde}
	\widetilde L(t)\geq {\rm e}^{-(A_1+\epsilon)t^\gamma (h(t))^{1-\gamma}}.
	\end{equation}
	Next, take $B_1>A_1$ and write
\begin{equation}\label{eq:taub-1}
\int_0^{B_1t^{\gamma-1}(h(t))^{1-\gamma}} {\rm e}^{-tx}\rho(x)\,{\rm d}x = \widetilde L(t) -
\int_{B_1t^{\gamma-1}(h(t))^{1-\gamma}}^\infty   {\rm e}^{-tx}\rho(x)\,{\rm d}x.
\end{equation}
The left-hand side of \eqref{eq:taub-1} is not bigger than
\[
\rho(B_1t^{\gamma-1}(h(t))^{1-\gamma}) \int_0^{B_1t^{\gamma-1}(h(t))^{1-\gamma}} {\rm e}^{-tx}{\rm d}x =
\rho(B_1t^{\gamma-1}(h(t))^{1-\gamma}) \frac{1- {\rm e}^{-B_1t^\gamma(h(t))^{1-\gamma}}}{t}.
\]
Moreover,
\begin{eqnarray*}
\int_{B_1t^{\gamma-1}(h(t))^{1-\gamma}}^\infty   {\rm e}^{-tx}\rho(x)\,{\rm d}x&=&
\int_{B_1t^{\gamma-1}(h(t))^{1-\gamma}}^\infty   {\rm e}^{-t\epsilon x}{\rm e}^{-tx(1-\epsilon)}\rho(x)\,{\rm d}x\\
&\leq & {\rm e}^{-(1-\epsilon)B_1t^\gamma (h(t))^{1-\gamma}}\widetilde L(\epsilon t),
\end{eqnarray*}
which in the light of \eqref{eq:el-tilde} yield
\begin{eqnarray*}
\rho(B_1t^{\gamma-1}(h(t))^{1-\gamma}) \frac{1- {\rm e}^{-B_1t^\gamma(h(t))^{1-\gamma}}}{t}&\geq& {\rm e}^{-(A_1+\epsilon)t^\gamma (h(t))^{1-\gamma}}-\widetilde L(\epsilon t){\rm e}^{-B_1(1-\epsilon)t^\gamma h(t)^{1-\gamma}}\\
&=& 	{\rm e}^{-(A_1+\epsilon)t^\gamma (h(t))^{1-\gamma}}\left(1- \widetilde L(\epsilon t){\rm e}^{t^\gamma h(t)^{1-\gamma}[-B_1(1-\epsilon)+(A_1+\epsilon)]}\right)\\
&=& 	{\rm e}^{-(A_1+\epsilon)t^\gamma (h(t))^{1-\gamma}}(1+o(1)),\quad t\to\infty,
	\end{eqnarray*}
provided $\epsilon<\frac{B_1-A_1}{B_1+1}.$
It follows
\begin{eqnarray*}
\log \rho(B_1t^{\gamma-1}(h(t))^{1-\gamma})&\geq& -(A_1+\epsilon)t^\gamma (h(t))^{1-\gamma}\\
&& +\log(1+o(1)) +\log t -\log(1-{\rm e}^{B_1t^\gamma (h(t))^{1-\gamma}})
\end{eqnarray*}

As mentioned above, $t^\gamma (h(t))^{1-\gamma}\to\infty$ as $t\to\infty,$ and consequently, for any $B_1>A_1$ and $\epsilon$ sufficiently small,
\begin{equation}\label{eq:taub-lower-2}
 \liminf_{t\to\infty}\frac{\log \rho(B_1t^{\gamma-1}(h(t))^{1-\gamma})}{t^\gamma (h(t))^{1-\gamma}}\geq -(A_1+\epsilon).
 \end{equation}
 Now, the number $\epsilon$ on the right-hand side can be sent to zero and eliminated.

To conclude the proof of part (i), substitute $x=B_1t^{\gamma-1}(h(t))^{1-\gamma}$. We need to write $t^\gamma (h(t))^{1-\gamma}$ as a function of $x.$
Recall that
\[t=x_{t}^{d+\alpha} g(x_t^\alpha) = x_t^{d+\alpha}h(t),\quad \mbox{i.e.}\quad \frac{h(t)}{t}= x_t^{-(d+\alpha)}.\]
It means
\[x=B_1\left(\frac{h(t)}{t}\right)^{1-\gamma}=\frac{B_1}{x_t^\alpha}. \]
Consequently,
\[t^\gamma (h(t))^{1-\gamma}= \frac{tx}{B_1}=\frac{t}{x_t^\alpha}=x_t^d g(x_t^\alpha) = \frac {B_1^{d/\alpha}}{x^{d/\alpha}}  g(\frac{B_1}{x}). \]
Assertion \eqref{eq:taub-lower} follows.

 (ii)  Assume now that \eqref{eq:taub-upper-assump} holds.  By an argument identical as above, \eqref{eq:taub-upper-assump} is satisfied for $\widetilde L(t).$
 Then for any $\epsilon>0$ there is $t_\epsilon>0$ such that for $t>t_\epsilon$
 \[\widetilde L(t)\leq {\rm e}^{-(A_2-\epsilon)t^\gamma  (h(t))^{1-\gamma}},\]
and on the other hand, with any $B_2>0,$
\[
\widetilde L(t)\geq \int_{B_2t^{\gamma-1}(h(t))^{1-\gamma}}^\infty {\rm e}^{-tx}{\rho(x)\,{\rm d}x \geq \rho(B_2t^{\gamma-1}(h(t))^{1-\gamma})\cdot \frac{1}{t} {\rm e}^{-B_2t^\gamma (h(t))^{1-\gamma}}}.\]
Using both these inequalities, taking logarithm and rearranging we arrive at
\[\frac{\log\rho(B_2t^{\gamma-1}(h(t))^{1-\gamma})}{t^\gamma (h(t))^{1-\gamma}} \leq \frac{\log t}{t^\gamma (h(t))^{1-\gamma}}-(A_2-\epsilon-B_2),\quad t>t_\epsilon\]
and further
\begin{equation}\label{eq:taub-upper-1}
\limsup_{t\to\infty}\frac{\log\rho(B_2t^{\gamma-1}(h(t))^{1-\gamma})}{t^\gamma (h(t))^{1-\gamma}}\leq -(A_2-\epsilon-B_2),
\end{equation}
which yields a viable result for $B_2<A_2$ (again, $\epsilon$ can be eliminated).

To conclude  the proof, similarly as in part (i), we substitute $x=B_2t^{\gamma-1}(h(t))^{1-\gamma}$, getting
\[t^\gamma (h(t))^{1-\gamma}= \frac{tx}{B_2}=\frac{t}{x_t^\alpha}=x_t^d g(x_t^\alpha) = \frac {B_2^{d/\alpha}}{x^{d/\alpha}} g(\frac{B_2}{x}). \]
This gives \eqref{eq:taub-upper} and completes the proof.
\end{proof}
\begin{remark}
	{\rm In the exponential Tauberian theorems without a correction term (cf. \cite[Theorem 2.1]{bib:F}), one was able to handle constants $B_1$ and $B_2.$ Now we do not know, in general, what the function $g$ looks like and how  it behaves asymptotically. We will be able to get rid of those constants in particular cases  only. }
\end{remark}

\subsection{Asymptotics of the IDS}
Finally,  applying the Tauberian theorem from the previous section, we give the formal proof of Lifshitz tail asymptotics of the IDS in Theorem \ref{th:IDS-asymp} stated in the Introduction.

\begin{proof}[Proof of Theorem \ref{th:IDS-asymp}]
The result follows directly from Theorems \ref{th:upper-short}, \ref{th:lower} and \ref{th:tauberian}  with
$g(x)=\log \frac{1}{F_q(D_0/x)},$  $j(x)=x^{d+\alpha}g(x^{\alpha}),$
$x_t=j^{-1}(t),$ and $h(t)=g(x_t^\alpha)$ as in Section \ref{sec:rate_der}.
\end{proof}

We complement our presentation with  less precise statements (the `loglog' regime), matching the usual statement of the  Lifshitz tail sometimes  found in the literature.

First we show that the behavior of $\log g(x)$, $\log x_t$ and $\log h(t)$  (see Section \ref{sec:rate_der})  at infinity are closely related.

\begin{lemma}\label{lem:h-and-g}  Let $g(x)=\log \frac{1}{F_q(D_0/x)},$  $j(x)=x^{d+\alpha}g(x^{\alpha}),$
	$x_t=j^{-1}(t),$ and $h(t)=g(x_t^\alpha)$.
The following three conditions are equivalent:
\begin{itemize}
\item[(i)]  $\lim_{x\to\infty}\frac{\log g(x)}{\log x} $ exists and is equal to $a\in[0,\infty];$

\item[(ii)]  $\lim_{t\to\infty}\frac{\log x_t}{\log t} $ exists and is equal to $b\in[0,1/(d+\alpha)];$

\item[(iii)] $\lim_{t\to\infty}\frac{\log h(t)}{\log t} $ exists and is equal to $c\in[0,1].$
\end{itemize}

\noindent Numbers $a$, $b$ and $c$ are related through
$$
b = \frac{1}{d+(a+1)\alpha} \quad \text{and} \quad c = 1 - (d+\alpha)b = 1- \frac{d+\alpha}{d+(a+1)\alpha}
$$
(here we use the standard convention $1/+\infty = 0$ and $1/0^{+} = + \infty$).
\end{lemma}

\begin{proof}
Recall that by \eqref{eq:t-and-xt}
\begin{equation}
t= x_{t}^{d+\alpha}g(x_t^\alpha)= x_t^{d+\alpha}h(t), \quad  t \geq t_0.
\end{equation}
In particular,
\begin{align} \label{eq:first_aux}
\log h(t)=\log t-(d+\alpha)\log x_t, \quad  t \geq t_0,
\end{align}
and
\begin{align} \label{eq:second_aux}
\frac{\log g(x_t^{\alpha})}{\log x_t^{\alpha}} = \frac{\log h(t)}{\log x_t^{\alpha}} = \frac{1}{\alpha} \frac{\log t}{ \log x_t} - \frac{d+\alpha}{\alpha}, \quad  t \geq t_0.
\end{align}
It follows directly from \eqref{eq:first_aux} that (ii)  and (iii) are equivalent and $c= 1 - (d+\alpha)b$ (or, equivalently, $b = (1-c)/(d+\alpha)$).
Moreover, by \eqref{eq:second_aux} and by the fact that $x_t \to \infty$ as $t \to \infty$, we see that (i) implies (ii) and then $b = 1/(d+(a+1)\alpha)$
(In particular, (i) gives (iii) with $c = 1-(d+\alpha)/(d+(a+1)\alpha).$) The converse implication (ii) $\Rightarrow$ (i) also follows from \eqref{eq:second_aux} by the fact that $[t_0, \infty) \ni t \mapsto x_t^{\alpha}$ is a continuous and increasing function onto $[x_{t_0}^{\alpha}, \infty)$.
\end{proof}

\begin{corollary}\label{coro:loglog}
Suppose that {\bf (B)}, {\bf (Q)},  and {\bf(W)} hold true. If $\lim_{x\to\infty}\frac{\log g(x)}{\log x}$ exists,
then
$$
\lim_{x\to 0^+}\frac{\log|\log\ell(x)|}{\log x}= -\frac{d}{\alpha}- \lim_{x\to \infty}\frac{\log g(x)}{\log x}
$$
and
$$
\lim_{t\to\infty}\frac{\log|\log L(t)|}{\log t} = 1- \frac{\alpha}{d+\left(1+\lim_{x\to \infty}\frac{\log g(x)}{\log x}\right)\alpha}.
$$
In particular, when $g(x)$ is of order lower than power-law (i.e. $\lim_{x\to \infty}\frac{\log g(x)}{\log x} = 0$), then
$$
\lim_{x\to 0^+}\frac{\log|\log\ell(x)|}{\log x}= -\frac{d}{\alpha}
$$
and
$$
\lim_{t\to\infty}\frac{\log|\log L(t)|}{\log t} = \frac{d}{d+\alpha}.
$$
\end{corollary}	

\begin{proof}
The assertion for the IDS follows directly from the estimates in Theorem \ref{th:IDS-asymp} and the definition of $\limsup$ and $\liminf$. For a proof of the second assertion, for $L(t)$, observe that by Theorems \ref{th:upper-short} and \ref{th:lower}, the definition of $\limsup$ and $\liminf$, and \eqref{eq:first_aux}, we have
$$
\lim_{t\to\infty}\frac{\log|\log L(t)|}{\log t} = \frac{d}{d+\alpha}+\frac{\alpha}{d+\alpha}\lim_{t\to\infty}\frac{\log h(t)}{\log t} = 1- \alpha \lim_{t \to \infty} \frac{\log x_t}{\log t}.
$$
An application of Lemma \ref{lem:h-and-g} completes the proof.
\end{proof}

\section{Discussion and examples}\label{sec:examples}
 We now discuss several specific classes of distributions $F_q$ to which our results apply directly.
Recall the notation: $g(x)=\log \frac{1}{F_q(D_0/x)},$  $j(x)=x^{d+\alpha} g(x^{\alpha}), $
$x_t=j^{-1}(t),$ $h(t)=g(x_t^\alpha).$ For more clarity, our discussion will be divided into four subsections.

\subsection{ Distribution functions $F_q$ with an atom at zero} \label{ex:atom}

Suppose there exists $\kappa_0 >0$ such that $F_q$ is continuous on $[0,\kappa_0]$ and $F_q(0) > 0$.  Then there are constants $C, \widetilde C>0$ such that
\begin{equation}\label{eq:stat-atom}
-C\leq \liminf_{ \lambda \searrow 0}\lambda^{d/\alpha}\log \ell(\lambda)\leq \limsup_{\lambda \searrow 0}\lambda^{d/\alpha}\log \ell(\lambda)\leq -\widetilde C\quad\mbox{ and }\quad \lim_{\lambda \searrow 0}\frac{\log|\log \ell(\lambda)|}{\log \lambda } =  - \frac{d}{\alpha}.
\end{equation}
Note that in this case we simply have $g(x) \asymp 1$ and $j(x) \asymp x^{d+\alpha}$ for large $x$, and therefore
\[x_t=j^{-1}(t) \asymp t^{\frac{1}{d+\alpha}} \qquad \text{and} \qquad
h(t) \asymp 1 ,\quad t\to\infty.\]
 In \cite{bib:KK-KPP-alloy-stable} we used Sznitman's coarse-graining method (the `enlargement of obstacles method') to derive the Lifschitz tail   in this case -   for alloy-type potentials with random variables $q_{\mathbf i}$ having an atom at 0.  The paper was concerned primarily with  $\Phi(\lambda)=\lambda^{\alpha/2}$, $\alpha \in (0,2]$ (i.e. with the fractional powers of the Laplace operator and the Laplace operator itself) - in this case we  were able to prove the existence of the limit $\lim_{\lambda \searrow 0}\lambda^{d/\alpha}\log \ell(\lambda)$  and to derive its actual value.
The value of this limit was coherent with that obtained for Poisson-type potentials in \cite{bib:Oku,bib:Szn1}.
The method of \cite{bib:KK-KPP-alloy-stable} is also suitable to cover the case of some other subordinate processes, but with no
precise scaling of principal Dirichlet eigenvalues at hand,  in general we would be able to obtain only
the statements for the $\limsup$ and $\liminf,$ exactly as in \eqref{eq:stat-atom} (cf. \cite{bib:KK-KPP2}).

\subsection{Distribution functions $F_q$ with polynomial decay at zero} \label{ex:log-rate}
 This section consists of two parts.

\noindent
(1) 	Suppose that there exist $\gamma_1, \gamma_2 > 0$, $\kappa_0>0$ and  constants $B_1, B_2>0$ such that
$$
B_1\kappa^{\gamma_1}\leq F_q(\kappa)\leq B_2\kappa^{\gamma_2}, \quad \kappa \in [0,\kappa_0].
$$
This example covers all absolutely continuous distributions whose densities near zero behave polynomially or explode at most logarithmically fast  (e.g. uniform, exponential, one-side normal, Weibull, arcsin, and many other distributions). In this case,
$$
 g(x)=\log \frac{1}{F_q(D_0/x)} \asymp \log x, \quad \text{for large \ $x.$}
$$
We then have
\[  j(x)= x^{d+\alpha} g(x^{\alpha}) \asymp x^{d+\alpha} \log x,\quad x\to\infty,\]
giving
\[x_t= j^{-1}(t) \asymp \left(\frac{t}{\log t}\right)^{\frac{1}{d+\alpha}},\quad t\to\infty\]
and
\[h(t)=  g(x_t^{\alpha}) \asymp  \frac{\alpha}{d+\alpha}\log t ,\quad t\to\infty\]
i.e. for some constants $C, \widetilde C>0$
\[-C \leq\liminf_{t\to\infty}\frac{\log L(t)}{t^{\frac{d}{d+\alpha}}(\log t)^{\frac{\alpha}{d+\alpha}}}\leq
\limsup_{t\to\infty}\frac{\log L(t)}{t^{\frac{d}{d+\alpha}}(\log t)^{\frac{\alpha}{d+\alpha}}}\leq -\widetilde C.\]
 Finally, for certain constants $C,\widetilde C>0$,  \[-C\leq \liminf_{\lambda \searrow 0}\frac{\lambda^{d/\alpha}}{\log \lambda} \log \ell(\lambda)\leq \limsup_{\lambda \searrow 0}\frac{\lambda^{d/\alpha}}{\log \lambda} \log \ell(\lambda)\leq - \widetilde C\quad\mbox{ and }\quad \lim_{\lambda \searrow 0}\frac{\log|\log \ell(\lambda)|}{\log \lambda } =  - \frac{d}{\alpha}. \]

\smallskip

\noindent
(2) The asymptotics of the IDS in the loglog regime has been  previously  established    by Kirsch ans Simon in \cite{bib:KS}  for  random Schr\"odinger operators $-\Delta+\sum_{{\bf i} \in \Z^d} q_{{\bf i}}(\omega) W(x-{\bf i})$ with bounded random variables $q_{{\bf i}}$ satisfying the  one-sided bound  $B_1\kappa^{\gamma_1}\leq F_q(\kappa),$ under  somewhat different assumptions on the single-site potential $W .$ Observe that such a one-sided bound is  not sufficient for
determining the term $h(t)$ needed in the `log' regime,  even asymptotically: for example, when there is an atom at zero (cf. Section \ref{ex:atom} above), then the one-sided bound holds, but
$h(t)\asymp 1$, $t \to \infty,$  while still $F_q(\kappa)\geq \mathbb Q[q=0]\geq B_1\kappa^{\gamma_1},$ $\kappa\leq \kappa_0$ - which should be contrasted with the results from part (1) above.

  Our present approach generalizes the results of Kirch and Simon: we are  able to derive
the 'loglog statements' for  both the integrated density of states and its Laplace transform from the more delicate statements in the `log' regime.
Indeed, as in this case  there is a constant $c>0$ such that
\[g(x)=\log \frac{1}{F_q(D_0/x)}\leq c \log x, \quad \mbox{for large $x$},
\]
it follows that
\[\lim_{x\to \infty}\frac {\log g(x)}{\log x} = 0,\]
and from Corollary \ref{coro:loglog} and Lemma \ref{lem:h-and-g} we get
 \[\lim_{\lambda \searrow 0} \frac{\log |\log\ell(\lambda)|}{\log \lambda}= -\frac{d}{\alpha}	 \qquad \mbox{and} \qquad
\lim_{t\to \infty}\frac{\log|\log L(t)|}{\log t}= \frac{d}{d+\alpha}.\]

\noindent Note also that the result for the Laplace transform is new.	 

\subsection{Distribution functions $F_q$ with exponential decay at zero}
  We now give an example what can happen  when the decay of $F_q$ near zero is faster than polynomial. For a fixed $\gamma >0$ we let
\[F_q(\kappa) = {\rm e}^{-\frac{1}{\kappa^\gamma}},\quad \kappa>0.\]
We verify that in this case
$$g(x)=  (x/D_0)^{\gamma},$$
 and further
\[j(x)= D_0^{-\gamma} x^{d+\alpha(1+\gamma)} \;\mbox{ and }\; x_t= (D_0^\gamma t)^{\frac{1}{d+\alpha(1+\gamma)}},\]
which gives
\[h(t)= c t^{\frac{\alpha\gamma}{d+\alpha(1+\gamma)}}=ct^{1-\frac{d+\alpha}{d+\alpha(1+\gamma)}}.\]
Consequently, the rate of decay for the Laplace transform of the IDS is
\[t^{\frac{d}{d+\alpha}}(h(t))^{\frac{\alpha}{d+\alpha}} = t^{1-\frac{\alpha}{d+\alpha(1+\gamma)}}\]
and there exist two constants $C, \widetilde C>0$ for which
\[-C\leq \liminf_{\lambda \searrow 0} \lambda^{\frac{d}{\alpha}+\gamma}\log \ell(\lambda) \leq \limsup_{\lambda \searrow 0} \lambda^{\frac{d}{\alpha}+\gamma}\log \ell(\lambda) \leq -\widetilde C.\]
The loglog limit is
 \[\lim_{\lambda \searrow 0} \frac{\log |\log\ell(\lambda)|}{\log \lambda}= -\frac{d}{\alpha}-\gamma.	\]
It means that we observe an increase in the power of the exponent which is due to the fast decay of the cumulative distribution function of $q.$
It should be noted that when $\gamma\to\infty,$ then the rate of decay of the Laplace transform approaches $t,$ which is the upper bound for the rate possible.

\subsection{Distribution functions $F_q$ with double-exponential decay near zero}
From  \eqref{eq:t-and-xt}  we see that $h(t)/t= x_t^{-d-\alpha}  \to 0$ as $t\to\infty,$ therefore the rate
$$
t^{\frac{d}{d+\alpha}}(h(t))^{\frac{\alpha}{d+\alpha}} \qquad \mbox{with} \qquad h(t)\asymp t, \ \ \ \ t \to \infty,
$$
is never possible. However, it can happen that $\frac{\log h(t)}{\log t}\to 1$ as $t\to\infty,$ which is illustrated by this example.

 For the distribution whose CDF is given by
\[F_q[\kappa]= {\rm e}^{1-{\rm e}^{
\frac{1}{\kappa}}},\quad \kappa>0\]
we have
\[ g(x)= -\log F_q(D_0/x)= {\rm e}^{\frac{x}{D_0}}-1.\]
It then follows from Theorem \ref{th:IDS-asymp} that there exist constants $C, \widetilde C, D, \widetilde D >0$ such that
$$
-C\leq \liminf_{\lambda \searrow 0} \lambda^{d/\alpha}{\rm e}^{ -{D}/{\lambda}}\log \ell(\lambda) \qquad \mbox{and} \qquad
\limsup_{\lambda \searrow 0} \lambda^{d/\alpha}{\rm e}^{-{\widetilde D}/{\lambda}}\log \ell(\lambda)\leq -\widetilde C.
$$
Consequently, we do not have the usual Lifschitz tail; the rate of decay of $\ell(\lambda)$ to zero is double exponential:
 we see that
\[ \widetilde D\leq\liminf_{\lambda \searrow 0} \lambda \log|\log \ell(\lambda)| \leq \limsup_{\lambda \searrow 0} \lambda \log|\log \ell(\lambda)|\leq D.\]
 This justifies the name:  {\em super-Lifchitz} tail.

To determine the asymptotical rate for the Laplace transform $L(t)$ first observe that the function
$$
k(t):=\left(D_0 \log \left(\frac{t}{(D_0 \log t)^{(d+\alpha)/\alpha}} + 1\right) \right)^{1/\alpha}
$$
is the asymptotic inverse
of the function $j(x) = x^{d+\alpha} g(x^{\alpha}) = x^{d+\alpha}\big({\rm e}^{\frac{x^{\alpha}}{D_0}}-1\big)$  as $x\to+\infty$.  Therefore,
$$x_t \asymp k(t) \qquad \mbox{and} \qquad h(t)\asymp \frac{t}{(\log t)^{\frac{d}{\alpha}+1}}, \quad t \to \infty,
$$
resulting in the asymptotics
\[-C \leq\liminf_{t\to\infty}\frac{\log L(t)}{t/\log t}\leq \limsup_{t\to\infty}\frac{\log L(t)}{t/\log t}\leq -\widetilde{C}\]
and
\[\lim_{t\to\infty} \frac{\log|\log L(t)|}{\log t} =1.\]
  As the last remark observe that   the last 'log log assertion' also follows directly from Corollary \ref{coro:loglog}, without the prior knowledge of the asymptotic behavior of the functions $x_t$ and $h(t)$.


\begin{thebibliography}{99}

\bibitem{bib:BP} M. Benderskii, L. Pastur: {\em On the spectrum of the one-dimensional Schr\"odinger equation with random potential}, Mat. Sb. 82 (1970) 245-256.

\bibitem{bib:BB} K. Bogdan, T. Byczkowski: {\em
Potential theory for the $\alpha$-stable Schr\"odinger operator on bounded Lipschitz domains}, Studia Math. 133 (1999) 53-92.

\bibitem{bib:Bou-Ken} J. Bourgain, C.\,E. Kenig: {\em On localization in the continuous Anderson-Bernoulli
model in higher dimension} Invent. Math. 161 (2005), 389-426.

\bibitem{bib:Car-Lac}  R. Carmona, J. Lacroix: {\em Spectral theory of random Schr\"{o}dinger operators}. Probability and its Applications. Birkh\"auser, Boston, Inc., Boston, MA, 1990.

\bibitem{bib:CMS}
R. Carmona, W.C. Masters, B. Simon:
\emph{Relativistic Schr\"odinger operators: asymptotic behaviour of the eigenfunctions}, J. Funct. Anal.
91, 1990, 117-142.

\bibitem{bib:Cha-Uri}
L. Chaumont, G. Uribe Bravo:
\emph{Markovian bridges: weak continuity and pathwise constructions},
Ann. Probab. 39 (2), 2011, 609-647.

\bibitem{bib:CS}
Z.-Q. Chen, R. Song:
\emph{Two-sided eigenvalue estimates for subordinate processes in domains},
J. Funct. Anal. 226 (2005) 90--113.

\bibitem{bib:Com-His} J.\,M. Combes, P.\,D. Hislop:  {\em Localization for some continuous, random Hamiltonians
in d-dimensions.}  J. Funct. Anal. 124, 149-180 (1994).

\bibitem{bib:DC}
M. Demuth, J.A. van Casteren:
\emph{Stochastic Spectral Theory for Self-adjoint Feller Operators. A Functional Analysis Approach}.
Birkh\"auser, Basel 2000.

\bibitem{bib:FL} R. Friedberg, J. Luttinger: {\em Density of electronic energy levels in disordered systems},
Phys. Rev. B 12 (1975) 4460-4474.

\bibitem{bib:F}
M. Fukushima:
\emph{On the spectral distribution of a disordered system and a range of a random walk},
Osaka J. Math. 11, 1974, 73-85.

\bibitem{bib:FNN} M. Fukushima, H. Nagai, and S. Nakao: {\em On an asymptotic property of spectra of a random
difference operator}, Proc. Japan Acad. 51 (1975) 100-102.

\bibitem{bib:Ger-His-Kle}  F. Germinet, P.\,D. Hislop, A.  Klein:  {\em Localization for Schr\"odinger operators with Poisson random potential.} J. Eur. Math. Soc. (JEMS) 9 (2007), no. 3, 577-607.

\bibitem{bib:J}
N. Jacob:
\emph{Pseudo-Differential Operators and Markov Processes: Markov Processes and Applications. Vol.\ I, II, III}.
Imperial College Press, London 2001--2005.

\bibitem{bib:KK-KPP1} K. Kaleta, K. Pietruska-Pa\l uba: {\em Integrated density of states for Poisson--Schr\"{o}dinger
perturbations of subordinate Brownian motions on the
Sierpi\'{n}ski gasket,} Stoch. Proc. Appl. 125 (2015), 1244-1281.

\bibitem{bib:KK-KPP2} K. Kaleta, K. Pietruska-Pa\l uba: {\em Lifschitz singularity for subordinate Brownian motions in presence of the Poissonian potential on the Sierpi\'{n}ski gasket}, Stoch. Proc. Appl. 128  (2018), 3897-3939.

\bibitem{bib:KK-KPP-alloy-stable} K. Kaleta, K. Pietruska-Pa\l uba: {\em Lifschitz tail for alloy-type models driven by the fractional Laplacian}, preprint 2019, available at arXiv:1906.03419.

\bibitem{bib:KSch2019} K. Kaleta, R.L. Schilling: {\em Progressive intrinsic ultracontractivity and heat kernel estimates for non-local Schr\"odinger operators}, preprint 2019, available at arXiv:1903.12004

\bibitem{bib:KM1} W. Kirsch, F. Martinelli: {\em On the density of states of Schr\"odinger operators with a random potential}, J. Phys. A 15 (1982) 2139-2156.

\bibitem{bib:KM2} W. Kirsch, F. Martinelli: {\em Large deviations and Lifshitz singularity of the integrated
density of states of random Hamiltonians}, Commun. Math. Phys. 89 (1983) 27-40.

\bibitem{bib:KS} W. Kirsch, B. Simon: \emph{Lifshitz Tails for Periodic Plus Random Potentials}, J. Stat. Phys. 42:5/6 (1986) 799-808.

\bibitem{bib:KV} W. Kirsch, I. Veseli\'c: {\em Lifshitz Tails for a Class of Schr\"odinger Operators with Random Breather-Type Potential},
Lett. Math. Phys. 94 (2010) 27-39.

\bibitem{bib:Klopp1} F. Klopp: {\em Weak Disorder Localization and Lifshitz Tails}, Commun. Math. Phys. 232 (2002) 125-155.

\bibitem{bib:Klopp2} F. Klopp: {\em Weak Disorder Localization and Lifshitz Tails: Continuous Hamiltonians}, Ann. Henri Poincar\'e 3 (2002) 711-737.

\bibitem{bib:Lif} I.M. Lifshitz, {\em Energy spectrum structure and quantum states of disordered condensed systems}, Soviet Physics Uspekhi {\bf 7} (1965), 549-573.

\bibitem{bib:Lut} J.M. Luttinger: {\em New variational method with applications to disordered systems}, Phys.
Rev. Lett. 37 (1976) 609-612.

\bibitem{bib:Mez} G. Mezincescu: {\em Bounds on the integrated density of electronic states for disordered Hamiltonians}, Phys. Rev. B 32 (1985) 6272-6277.

\bibitem{bib:Nag} H. Nagai: {\em On an exponential character of the spectral distribution function of a random difference operator}, Osaka J. Math. 14 (1977) 111-116.

\bibitem{bib:Nak} S. Nakao: {\em On the spestral distribution of the Schr\"{o}dinger operator with random potential}, Japan J. Math. Vol. 3 (1977), 111-139.

\bibitem{bib:Pas} L.A. Pastur: {\em The behavior of certain Wiener integrals as $t\to\infty$ and the density of states of Schr\"{o}dinger equations with random potential}, (Russian) Teoret. Mat. Fiz. 32 (1977) 88-95.

\bibitem{bib:KPP1}
K. Pietruska-Pa{\l}uba: {\em The Lifschitz singularity for the density of states on the Sierpi\'nski gasket},  Probab. Theory Related Fields 89 (1991), no. 1, 1-33.


\bibitem{bib:Oku} H. Okura: {\em On the spectral distributions of certain integro-differential operators with random potential}, Osaka J. Math. 16 (1979), no. 3, 633-666.

\bibitem{bib:RS}
M. Reed, B. Simon:
\emph{Methods on Modern Mathematical Physics. Vol. 4: Analysis of Operators}, Academic Press, 1978.

\bibitem{bib:RW} M. Romerio, W. Wreszinski: {\em On the Lifshitz singularity and the tailing in the density
of states for random lattice systems}, J. Stat. Phys. 21 (1979) 169-179.

\bibitem{bib:SSV}
R. Schilling, R. Song, Z. Vondra\v{c}ek:
\emph{Bernstein Functions}, Walter de Gruyter, 2010.

\bibitem{bib:Sat}
K. Sato:
\emph{L\'{e}vy Processes and Infinitely Divisible Distributions}.
Cambridge University Press, Cambridge 1999.

\bibitem{bib:Sim}
B. Simon:
\emph{Lifshitz tails for the Anderson model},
J. Stat. Phys. 38 (1985) 65-76.

\bibitem{bib:Szn1} A.S. Sznitman: \emph{Lifschitz tail and Wiener sausage I}, J. Funct. Anal. {\bf 94} (1990), 223-246.

\end{thebibliography}
 \end{document}